\documentclass[a4paper,12pt]{amsart}

\usepackage{amssymb, amsmath, amsthm, mathrsfs, braket, xspace}
\usepackage[margin=1.0in]{geometry}

\numberwithin{equation}{section}
\usepackage[colorlinks=true]{hyperref}

\usepackage{mathtools}
\mathtoolsset{showonlyrefs=true} 
\usepackage{enumitem}
\setlist[enumerate]{label=$(\mathrm{\arabic*})$}

\usepackage[all,2cell]{xy}
\objectmargin+{1mm}
\labelmargin+{0.8mm}
\SelectTips{cm}{12}
\usepackage{aliascnt}

\newtheorem{thm}{Theorem}[section]

\newaliascnt{cor}{thm}
\newtheorem{cor}[cor]{Corollary}
\aliascntresetthe{cor}

\newaliascnt{lem}{thm}
\newtheorem{lem}[lem]{Lemma}
\aliascntresetthe{lem}

\newaliascnt{prop}{thm}
\newtheorem{prop}[prop]{Proposition}
\aliascntresetthe{prop}

\newaliascnt{conj}{thm}
\newtheorem{conj}[conj]{Conjecture}
\aliascntresetthe{conj}

\theoremstyle{definition}
\newaliascnt{dfn}{thm}
\newtheorem{dfn}[dfn]{Definition}
\aliascntresetthe{dfn}

\newaliascnt{rem}{thm}
\newtheorem{rem}[rem]{Remark}
\aliascntresetthe{rem}

\newaliascnt{prob}{thm}

\aliascntresetthe{prob}

\newaliascnt{ex}{thm}
\newtheorem{ex}[ex]{Example}
\aliascntresetthe{ex}

\newcommand{\Q}{\mathbb{Q}}
\newcommand{\Qp}{\Q_p}
\newcommand{\R}{\mathbb{R}}
\newcommand{\C}{\mathbb{C}}
\newcommand{\Z}{\mathbb{Z}}
\newcommand{\Zhat}{\widehat{\Z}}

\newcommand{\bC}{\mathbb{C}}

\newcommand{\bF}{\mathbb{F}}

\newcommand{\sX}{\mathscr{X}}

\newcommand{\cO}{\mathcal{O}}
\newcommand{\cM}{\mathcal{M}}
\newcommand{\cN}{\mathcal{N}}

\newcommand{\fB}{\mathfrak{B}}

\DeclareMathOperator{\Br}{Br}
\DeclareMathOperator{\ch}{ch}
\DeclareMathOperator{\Coker}{Coker}
\DeclareMathOperator{\Cor}{Cor}

\DeclareMathOperator{\Gal}{Gal}

\DeclareMathOperator{\Hom}{Hom}

\DeclareMathOperator{\Jac}{Jac}
\DeclareMathOperator{\Ker}{Ker}
\DeclareMathOperator{\res}{res}
\DeclareMathOperator{\red}{red}

\DeclareMathOperator{\Spec}{Spec}

\DeclareMathOperator{\tr}{tr}

\newcommand{\ab}{\mathrm{ab}}
\newcommand{\cd}{\operatorname{cd}}

\newcommand{\Gm}{\mathbb{G}_{m}}

\newcommand{\otimesM}{\overset{M}{\otimes}}

\usepackage{color}

\title{Zero-cycles on varieties over a $\fB_s$-field}

\author[T. Hiranouchi]{Toshiro Hiranouchi}\address[T. Hiranouchi]{
Department of Basic Sciences, Graduate School of Engineering, 
Kyushu Institute of Technology, 
1-1 Sensui-cho, Tobata-ku, Kitakyushu-shi, 
Fukuoka, 804-8550 JAPAN}
\email{hira@mns.kyutech.ac.jp}

\author[R. Sugiyama]{Rin Sugiyama} \address[R. Sugiyama]{
Department of Mathematics, Physics and Computer Science, 
Japan Women's University, 2-8-1 Mejirodai, Bunkyo-ku, Tokyo, 112-8681 JAPAN
}\email{sugiyamar@fc.jwu.ac.jp}

\begin{document}

\begin{abstract}
A field $F$ is a $\fB_s$-field if, 
for every finite extension $E'/E$ of $F$, 
the norm map $K_s^M(E')\to K_s^M(E)$ 
of the Milnor $K$-groups is surjective. 
In particular, finite fields ($s=1$), local fields, and certain global fields (with $s=2$) satisfy this condition.
For such a field $F$ and a $d$-dimensional variety $X$ over $F$, 
we prove that $CH^{d+n}(X,n)$ is divisible for $n \geq s+1$. 
Under a suitable condition on the index of $X$, 
$CH^{d+s}(X,s)$ is isomorphic to the direct sum of the Milnor $K$-group $K_{s}^M(F)$ and a divisible group. 
As an application, we study the Kato homology groups $KH_0^{(n)}(X,\Z/l^r\Z)$ for any prime $l$ different from the characteristic of $F$. 
\end{abstract}

\maketitle

\section{Introduction}
\label{sec:intro}

The study of the higher Chow groups $CH^m(X,n)$ for a scheme $X$ introduced by Bloch is fundamental in understanding motivic cohomology and its applications to arithmetic geometry. 
When $X$ is a smooth projective and geometrically irreducible scheme over a finite field $F$ of $d:=\dim(X)$, 
the structure of the group $CH^{d}(X,0) = CH_0(X)$ is well-understood 
by the higher-dimensional class field theory 
(e.g.~\cite[Section 10, Proposition 9]{KS83b}). 
The theorem below describes the structure of $CH^{d+n}(X,n)$ for $n\ge 1$; 
in particular, $CH^{d+1}(X,1)$ is finite.

\begin{thm}[{\cite[Theorem~1.2 and Theorem~1.3]{Akh04c}}] 
\label{thm:Akh}
    Let $X$ be a smooth projective and geometrically irreducible scheme over a finite field $F$ of $d = \dim(X)$.
    Then, we have $CH^{d+n}(X,n) = 0$ for $n\ge 2$, and
    the structure morphism $f\colon X\to \Spec(F)$ induces an isomorphism
    \[ 
       f_{\ast}\colon CH^{d+1}(X,1) \xrightarrow{\simeq} CH^{1}(\Spec(F),1) \simeq F^\times.
    \]
\end{thm}

We investigate an analogue over fields $F$ satisfying the following condition:
for any finite extension $E/F$, 
and any finite extension $E'/E$, the norm map 
\[
N_{E'/E}\colon K_s^M(E') \to K_s^M(E)
\]
on the Milnor $K$-groups is surjective for some $s\ge 0$. 
Such fields are called \textbf{$\fB_s$-fields} in 
\cite[Section 3.3]{Kat78}
($C_0^s$ in the sense of \cite{KK86}, see \cite[Lemma 2]{KK86}). 
This condition is closely related to the cohomological dimension of $F$. 
For example, 
when the characteristic of $F$ is 0,
$F$ is a $\fB_s$-field if and only if the cohomological dimension of $F$ is at most $s$ (cf.~\autoref{prop:Kato}). 
A finite field is $\fB_1$, while
local fields, totally imaginary number fields, and function fields of one variable over a finite field are $\fB_2$
(\autoref{lem:global_B2}). 
Moreover, if $F$ is an \emph{$N$-dimensional local field} then $F$ is a $\fB_{N+1}$-field (cf.~\autoref{prop:Kato}). 
Here, an \textbf{$N$-dimensional local field} is a complete discrete valuation field whose residue field is an $(N-1)$-local field, and a $0$-dimensional local field is a finite field. 
The structure of $CH^{d+N}(X,N)$, often denoted by $SK_{N}(X)$, has been described by the class field theory 
for $X$ (\cite{Blo81}, \cite{Sai85a}, \cite{For16} and \cite{GKR22}). 
For a $\fB_s$-field $F$, 
we consider the higher Chow group $CH^{d+n}(X,n)$ for $n\ge s$. 
Our main result is the following.

\begin{thm}[{\autoref{thm:main}, \autoref{cor:main}}]
\label{thm:main_intro}
    Let $F$ be a $\fB_s$-field for some $s\ge 0$, and 
    let $X$ be a smooth projective and geometrically irreducible scheme over $F$ of $d = \dim(X)$. 
    Then, $CH^{d+n}(X,n)$ is divisible for $n\ge s+1$, and the structure morphism $f\colon X\to \Spec(F)$ induces a surjective homomorphism 
    \[
    f_\ast\colon CH^{d+s}(X,s) \twoheadrightarrow CH^{s}(\Spec(F), s).
    \]
	If we further assume that $X$ has a zero-cycle of degree one, then the kernel of $f_\ast$ is divisible.
\end{thm}

\begin{rem}
For example, $F = \bC$ is a $\fB_0$-field. 
For an elliptic curve $X=E$ over $\bC$, 
the kernel of $f_\ast \colon CH^{1}(E,0) = CH_0(E) \to CH^0(\Spec(\bC),0)  = \Z$ 
is isomorphic to $E(\bC)$. 
Since  $E(\bC)[m]\simeq (\Z/m\Z)^{\oplus 2}$, 
the group has torsion and hence the kernel $E(\bC)$ is divisible but not \emph{uniquely} divisible. 
This example suggests that the kernel $\Ker(f_{\ast})$ in the above theorem is neither trivial nor uniquely divisible in general and 
our result is an optimal generalization of \autoref{thm:Akh}. 
\end{rem}
For the scheme $X$ over $F$ as in the above theorem, 
the Nesterenko-Suslin and Totaro theorem (\cite{NS89}, \cite{Tot92}) gives an isomorphism 
$CH^{s}(\Spec(F),s) \simeq K_{s}^M(F)$. 
Under the additional assumption that \(X\) has a zero-cycle of degree one,  
we obtain the following decomposition 
\[
    CH^{d+s}(X,s) \simeq K_{s}^M(F) \oplus D, 
\]
for a divisible group $D$.
Much is known about the group structure of $K_{s}^M(F)$ 
for \emph{arithmetic fields}, such as local fields, global fields, and higher-dimensional local fields. 
For example, when $F$ is a local field, 
Sivitskii's theorem asserts that $K_n^M(F)$ is uncountable uniquely divisible for $n\ge 3$ 
(\cite[Chapter~IX, Theorem~4.11]{FV02}) 
and Moore's theorem (\cite[Chapter IX, Theorem 4.3]{FV02}) 
describes the structure of $K_2^M(F)$. 
\autoref{thm:main_intro} implies that
$CH^{d+n}(X,n)$ is divisible for $n\ge 3$, 
and, under the additional
assumption that $X$ has a zero-cycle of degree one, that there  is a decomposition  
\[
 CH^{d+2}(X,2) \simeq \mu(F) \oplus D, 
\]
where $\mu(F)$ denotes the group of roots of unity in $F$ and $D$ is a divisible group (see \autoref{cor:CH}). 
These results are refinements of \cite[Theorem 1.7 (3), Corollary 1.8]{GKR22}. 
For other arithmetic fields $F$, the structure of $CH^{d+n}(X,n)$ can be described via the above theorem 
(see \autoref{cor:CH_high}, and \autoref{cor:CH_FF}).

As another application of our main theorem, 
we study the Kato homology groups of $X$. 
To state the results precisely, 
let $F$ be a field of characteristic $p\ge 0$. 
For an excellent scheme $X$ over $F$ 
and a prime $l$ with $l\neq p$,
there is a homological complex $KC^{(n)}_{\bullet}(X,\Z/l^r\Z)$ of Bloch--Ogus type (\cite[Proposition~1.7]{Kat86}): 
\[
\xymatrix@C=5mm@R=0mm{
 \cdots \ar[r]^-{\partial} & \displaystyle\bigoplus_{x\in X_{j}}H^{n+j+1}_{l^r}(F(x)) \ar[r]^-{\partial} &\cdots  
 \ar[r]^-{\partial}& \displaystyle\bigoplus_{x\in X_{1}}H^{n+2}_{l^r}(F(x)) \ar[r]^-{\partial}  & \displaystyle\bigoplus_{x\in X_{0}}H^{n+1}_{l^r}(F(x)). \\
 & \mbox{\small degree $j$} & & \mbox{\small degree $1$} & \mbox{\small degree $0$} 
}
\]
for any $n\ge 0$, 
where $X_j$ is the set of points $x$ in $X$ with $\dim(\overline{\set{x}}) = j$,  
$F(x)$ is the residue field at $x$, and
\[
H_{l^r}^{m+1}(F(x)) := H^{m+1}(F(x), \mu_{l^r}^{\otimes m}).
\]
The \textbf{Kato homology group} of $X$ (with coefficients in $\Z/l^r\Z$) is defined 
to be the homology group 
\[
KH_j^{(n)}(X,\Z/l^r\Z) := H_j(KC_{\bullet}^{(n)}(X,\Z/l^r\Z))
\]
of the above complex $KC_\bullet^{(n)}(X,\Z/l^r\Z)$. 

For the induced map $f_{\ast}:KH_0^{(n)} (X,\Z/l^r\Z) \to KH_0^{(n)}(\Spec(F),\Z/l^r\Z)$ 
arising from the structure morphism $f\colon X\to \Spec(F)$, 
\autoref{thm:Akh} and \autoref{thm:main} yield the following. 

\begin{cor}[{\autoref{thm:KH}}]
    Let $F$ be a $\fB_s$-field for some $s\ge 1$ of characteristic $p \ge 0$, and 
    let $X$ be a smooth projective and geometrically irreducible scheme over $F$ 
    with structure morphism $f\colon X\to \Spec(F)$. 
    Assume that $X$ has a zero-cycle of degree one.
    Then, 
   \[
    f_{\ast}\colon KH_0^{(s-1)}(X,\Z/l^r\Z) \xrightarrow{\simeq} KH_0^{(s-1)}(\Spec(F),\Z/l^r\Z)
    \]
    is an isomorphism for any prime $l\neq p$ and any $r\ge 1$.
\end{cor}

Applying the above corollary to a function field $F$ of one variable over a finite field, 
and to the local field $F_v$ associated with a valuation $v$ of $F$, we obtain the following isomorphisms.
\begin{align*}
    KH_0^{(1)}(X,\Z/l^r\Z) &\simeq KH_0^{(1)}(\Spec(F),\Z/l^r\Z) \simeq \Br(F)[l^r],\ \mbox{and}\\
    KH_0^{(1)}(X_v,\Z/l^r\Z) &\simeq KH_0^{(1)}(\Spec(F_v),\Z/l^r\Z) \simeq \Br(F_v)[l^r], 
\end{align*}
where $X_v = X\otimes_F F_v$ and $\Br(F)[l^r]$ is the $l^r$-torsion subgroup of the Brauer group $\Br(F)$ of $F$.
The Hasse--Brauer--Noether theorem on the Brauer group then yields the following short exact sequence: 
\[
0 \to KH_0^{(1)}(X,\Z/l^r\Z) \to \bigoplus_v KH_0^{(1)}(X_v,\Z/l^r\Z) \to \Z/l^r\Z \to 0.
\]
The existence of this short exact sequence is a part of the Kato conjectures (cf.~\autoref{conj:Kato3}). 
\subsection*{Notation}
Throughout this note, 
by a \textbf{global field} we mean a finite extension of $\Q$ (a \textbf{number field}) or a 
function field of one variable over a finite field. 
For an abelian group $G$ and $m\in \Z_{\ge 1}$,  
we write $G[m]$ and $G/m$ for the kernel and cokernel of the multiplication by $m$ on $G$ respectively. 
For a field $F$, we denote by $\ch(F)$ its characteristic 
and we denote by $\cd(F)$ (resp.~$\cd_l(F)$) the cohomological  dimension (resp.~$l$-cohomological dimension for a prime number $l$) of the absolute Galois group $G_F$ of $F$ (cf.~\cite[Chapter I, Section 3]{SerreCG}).

\subsection*{Acknowledgements} 
The first author would like to thank Jitendra Rathore for sending him his preprint \cite{GR25} and providing important hints regarding the treatment of $V(X)$ for a higher-dimensional scheme $X$.
The first author was supported by JSPS KAKENHI Grant Number 24K06672.
The second author was supported by JSPS KAKENHI Grant Number 21K03188.

\section{Preliminaries}
\label{sec:Mackey}
This section establishes the foundational concepts essential to the paper's main arguments. 
It begins by defining the Mackey product following \cite[Section~5]{Kah92a} (see also \cite[Remark~1.3.3]{IR17}, \cite[Section~3]{RS00}). 
Following this, it formally introduces higher Chow groups and concludes by defining a $\fB_s$-field and detailing its fundamental properties, which are crucial for the subsequent sections.

\subsection*{Mackey product}
A \textbf{Mackey functor} $\cM$ over a field $F$ 
(a cohomological finite Mackey functor over $F$ in the sense of \cite{Kah92a}) is a covariant functor $\cM$ 
from the category of field extensions of $F$ to the category of abelian groups equipped with a contravariant structure for finite extensions over $F$ 
satisfying some conditions (see \cite{Kah92a}, or \cite{Hir24} for the precise definition). 
The category of Mackey functors forms a Grothendieck abelian category (cf.~\cite[Appendix A]{KY13}), 
so any morphism of Mackey functors $f\colon \cM\to \cN$, 
that is, a natural transformation, 
gives the image $\operatorname{Im}(f)$, the cokernel $\Coker(f)$ and so on.

\begin{ex}\label{ex:MFexs}
	\begin{enumerate}
	\item 
	For any endomorphism $f\colon \cM\to \cM$ of a Mackey functor $\cM$, 
	we denote the cokernel by 
	$\cM/f := \Coker(f)$. 
	This Mackey functor is given by 
	\[
	(\cM/f)(E) = \Coker\left(f(E)\colon \cM(E)\to \cM(E)\right).
	\]

	\item 
	Any commutative algebraic group $G$ over $F$ induces a Mackey functor, 
    also denoted by $G$, by defining $G(E)$ for any field extension $E/F$
	 (cf.\ \cite[(1.3)]{Som90}, \cite[Proposition~2.2.2]{IR17}). 
	In particular, 
	the multiplicative group $\Gm$ is a Mackey functor given by 
	$\Gm(E) = E^{\times}$ 
	for any field extension $E/F$. 
	The translation maps are the norm 
	$N_{E'/E}\colon (E')^{\times} \to E^{\times}$ and the inclusion $E^{\times} \to (E')^{\times}$. 
	 \item 
 Recall that, for $n\ge 1$, the \textbf{Milnor $K$-group} $K_n^M(F)$ of a field $F$   
	 is the quotient group of $(F^{\times})^{\otimes_{\Z} n}$ by the subgroup generated by all elements of the form 
	 $a_1\otimes \cdots \otimes a_n$ 
	 with $a_i + a_{j} = 1$ for some $i\neq j$. 
	 We also put $K_0^M(F) = \Z$. 
	 For an extension $E'/E$ over $F$, 
	 the restriction map $\res_{E'/E}:K_n^M(E)\to K_n^M(E')$ 
	 and the norm map $N_{E'/E}:K_n^M(E')\to K_n^M(E)$ when $[E':E]<\infty$ 
	 gives the structure of the Mackey functor $K_n^M$.
	\end{enumerate}
\end{ex}

\begin{dfn}\label{def:otimesM}
    For Mackey functors $\cM_1,\ldots, \cM_n$ over $F$, 
    the \textbf{Mackey product} $\cM_1\otimesM \cdots \otimesM \cM_n$ is defined as follows: For any field extension $F'/F$, 
    \begin{equation}\label{eq:otimesM}
		(\cM_1 \otimesM \cdots \otimesM \cM_n ) (F') := 
		\left.\left(\bigoplus_{E/F':\,\mathrm{finite}} \cM_1(E) \otimes_{\Z} \cdots \otimes_{\Z} \cM_n(E)\right)\middle/\ (\textbf{PF}),\right. 
    \end{equation}
    where (\textbf{PF}) stands for the subgroup generated 
	by elements of the following form: 
	For finite field extensions $F' \subset E \subset E'$, 
	\[
		 x_1 \otimes \cdots\otimes \tr_{E'/E}(\xi_{i_0}) \otimes \cdots \otimes x_{n} - \res_{E'/E}(x_1) \otimes \cdots \otimes \xi_{i_0} \otimes \cdots \otimes \res_{E'/E}(x_n) \quad 
	\]
	for $\xi_{i_0} \in \cM_{i_0}(E')$ and $x_{i} \in \cM_i(E)$ for $i\neq i_0$.
\end{dfn}

For the Mackey product $\cM_1\otimesM \cdots \otimesM \cM_n$,
we write $\set{x_1,\ldots ,x_n}_{E/F'}$ 
for the image of $x_1 \otimes \cdots \otimes x_n \in \cM_1(E) \otimes_{\Z} \cdots \otimes_{\Z} \cM_n(E)$ in the product
$(\cM_1\otimesM \cdots \otimesM \cM_n)(F')$. 
The subgroup (\textbf{PF}) gives a relation of the form 
\begin{equation}
    \label{eq:pf}
\set{x_1, \ldots, \tr_{E'/E}(\xi_{i_0}), \ldots,x_{n}}_{E/F'} = \set{x_1,\ldots, \xi_{i_0},\ldots,x_n}_{E'/F'}
\end{equation}
for $\xi_{i_0} \in \cM_{i_0}(E')$ and $x_{i} \in \cM_i(E)$ for $i\neq i_0$
omitting the restriction maps $\res_{E'/E}$ on the right.
The above equation \eqref{eq:pf} is referred to as \textbf{the projection formula}.

The product $\cM_1\otimesM \cdots \otimesM \cM_n$ is a Mackey functor 
with transfer maps defined as follows:
For a field extension $E/F$, 
the restriction 
$\res_{E/F}\colon (\cM_1\otimesM \cdots \otimesM \cM_n)(F) \to (\cM_1\otimesM \cdots \otimesM \cM_n)(E)$ is defined by 
\[
 	\res_{E/F}(\set{x_1,\ldots ,x_n}_{F'/F})= \sum_{i=1}^r e_i\set{\res_{E_i'/F'}(x_1), \ldots , \res_{E_i'/F'}(x_n)}_{E_i'/E}, 
\]
where $E\otimes_FF' = \bigoplus_{i=1}^r A_i$ for some local Artinian algebras $A_i$ 
 of dimension $e_i$ over the residue field $E_i' := A_i/\mathfrak{m}_{A_i}$.
When $E/F$ is finite, 
the transfer map
$\tr_{E/F}\colon  (\cM_1\otimesM \cdots \otimesM \cM_n)(E) \to (\cM_1\otimesM \cdots \otimesM \cM_n)(F)$ 
is given by 
\[
	\tr_{E/F}(\set{x_1,\ldots ,x_n}_{E'/E}) = \set{x_1,\ldots ,x_n}_{E'/F}.
\]
The product $\otimesM$ endows the category of Mackey functors with a tensor structure.

\subsection*{Higher Chow groups}
Let $F$ be a field, 
and let $X$ be a separated scheme of finite type over $F$ 
of $d = \dim(X)$.
The boundary maps on Milnor $K$-groups give rise to a complex (cf.~\cite[Section~1]{Kat86})
\begin{equation}
\label{eq:MKC2}
\vcenter{
\xymatrix@C=5mm@R=-2mm{
 \cdots \ar[r]^-{\partial} & \displaystyle\bigoplus_{x\in X_{j}}K_{n+j}^M(F(x)) \ar[r]^-{\partial} &\cdots  
 \ar[r]^-{\partial}& \displaystyle\bigoplus_{x\in X_{1}}K_{n+1}^M(F(x)) \ar[r]^-{\partial}  & \displaystyle\bigoplus_{x\in X_{0}}K_{n}^M(F(x)) \\
 & \mbox{\small degree $j$} & & \mbox{\small degree $1$} & \mbox{\small degree $0$} 
}}
\end{equation}
for any $n\ge 0$. 
The $0$-th homology group of the complex \eqref{eq:MKC2} is isomorphic to  
the higher Chow group of $X$ 
when $X$ is smooth, quasi-projective, and geometrically irreducible over $F$ 
(\cite[Theorem~3]{Kat86b}, \cite[Theorem~5.5]{Akh04}):
\begin{equation}\label{eq:MK-Chow} 
  \Coker\left(\partial\colon \bigoplus_{x\in X_1}K_{n+1}^M(F(x))\to \bigoplus_{x\in X_0}K_{n}^M(F(x))\right)
  \simeq CH^{d+n}(X,n). 
\end{equation}
The structure morphism $f\colon X\to \Spec(F)$ induces a map 
\begin{equation}\label{eq:fCH}
    f_{\ast}^{CH}\colon CH^{d+n}(X,n) \to CH^n(\Spec(F),n) \simeq K_n^M(F),  
\end{equation}
where the last isomorphism is the Nesterenko--Suslin/Totaro theorem (\cite{NS89}, \cite{Tot92}). 
The kernel is denoted by 
\[
A^{d+n}(X,n) := \Ker(f_{\ast}^{CH}).
\]

Define the Mackey functor $\underline{CH}^{d+n}(X,n)$ 
given by the higher Chow group 
\[
\underline{CH}^{d+n}(X,n)(E) = CH^{d+n}(X_E,n)
\]
for each extension $E/F$, 
where $X_E := X\otimes_F E$.
The transfer maps are given by the pull-back and the push-forward  
along $j:\Spec(E) \to \Spec(F)$.
The map $f_{\ast}^{CH}$ induces a map
\[
f_{\ast}^{CH}\colon \underline{CH}^{d+n}(X,n) \to \underline{CH}^n(\Spec(F),n) \simeq K_n^M
\]
 of Mackey functors 
 ({see, for example \cite[Proposition~7.6.1, Proposition~7.6.2]{Akh00}, and \cite[Section~4]{Akh04}}), 
and its kernel is denoted by $\underline{A}^{d+n}(X,n)$. 
In particular, the Chow group $CH^{d}(X,0) = CH_0(X)$ of zero-cycles on $X$ gives rise to 
the Mackey functors 
\[
\underline{CH}_0(X) := \underline{CH}^d(X,0),\ \mbox{and}\ \underline{A}_0(X) := \underline{A}^d(X,0).
\]

\begin{thm}[{\cite[Theorem~6.1]{Akh04}}]\label{thm:Akh2}
    Let $X$ be a smooth projective and geometrically irreducible scheme over a field $F$ of $d = \dim(X)$.
    There is an isomorphism 
    \[
    K(F;\underline{CH}_0(X),\overbrace{\Gm,\ldots,\Gm}^n) \xrightarrow{\simeq} CH^{d+n}(X,n),
    \]
    for $n\ge 0$, where 
    the left hand side is the Somekawa type $K$-group 
    which is a quotient of the Mackey product 
    $\Big(\underline{CH}_0(X)\otimesM \Gm^{\otimesM n}\Big)(F)$. 
\end{thm}

Combining \eqref{eq:MK-Chow} with the description of Milnor $K$-groups over global fields (\cite[Chapter II, Theorem 2.1]{BT73}), 
we obtain the following results. 
Here, we recall that a global field is a finite extension of $\Q$ or a 
function field of one variable over a finite field (cf.~Notation). 

\begin{prop}[{\cite[Corollary 7.2]{Akh04}}]\label{prop:Akh04_7.2}
    Let $F$ be a global field and 
    let $X$ be a smooth quasi-projective and geometrically irreducible scheme over $F$ of $\dim(X) = d$. 
    \begin{enumerate}
        \item If $\ch(F) = 0$, then we have 
    $CH^{d+n}(X,n) = 
    \begin{cases} 
    torsion, & n = 2,\\
    2\text{-}torsion & n\ge 3.
    \end{cases}$
    \item If $\ch(F) >0$, then 
    $CH^{d+n}(X,n) = 
    \begin{cases} 
    torsion, & n = 2,\\
    0 & n\ge 3.
    \end{cases}$
    \end{enumerate}
\end{prop}

\subsection*{\texorpdfstring{$\fB_s$}{Bs}-field}
Following \cite[Chapter II, Section 3.3]{Kat78}, we introduce the notion of a \emph{$\fB_s$-field} 
and its properties. 

\begin{dfn}[{\cite[Chapter II, Section 3.3, Definition 3]{Kat78}}]
\label{def:Bs}
For an integer $s\ge 0$, 
a field $F$ is called a \textbf{$\fB_s$-field} (resp.~\textbf{$\fB_s(l)$-field} for a prime $l$) 
if for any finite extension $E/F$ and any further finite extension $E'/E$, 
we have $K_s^{M}(E)/N_{E'/E}(K_s^M(E')) = 0$ (resp.~$(K_s^{M}(E)/N_{E'/E}K_s^M(E'))\{l\} = 0$, 
where 
 $G\{l\}$ is the $l$-primary part of an abelian group $G$). 
\end{dfn}

Note that the quotient group 
$K_s^{M}(E)/N_{E'/E}K_s^M(E')$ is torsion of finite exponent, since $[E':E] = N_{E'/E}\circ \res_{E'/E}$. 
Therefore, $K_s^{M}(E)/N_{E'/E}K_s^M(E')\{l\} = 0$ implies 
$(K_s^{M}(E)/N_{E'/E}K_s^M(E'))/l^r = 0$ for any $r\ge 1$ and hence 
the norm map 
\begin{equation}
\label{eq:N/l^r}
    N_{E'/E}\colon K_s^M(E')/l^r \to K_s^M(E)/l^r
\end{equation}
is surjective. 

Recall that 
a field $F$ is called a \textbf{$C_s$-field} if, for every homogeneous polynomial $f(x_1,\ldots,x_n)$ of degree $d$ 
with coefficients in $F$, the equation $f(x_1,\ldots, x_n) = 0$ has a non-trivial zero 
in $F^n$ if $n > d^s$ (cf.\ \cite[Chapter II, Section 4.5]{SerreCG}). 
By \cite[Chapter I, Section 4, Proposition 1]{Kat77}, 
a $C_2$-field is a $\fB_2$-field.
By the norm residue isomorphism theorem (the Bloch--Kato conjecture) (cf.~\cite[Theorem~6.16]{Vo11}, \cite{Wei09}), 
the property $\fB_s(l)$ is related to the cohomological $l$-dimension $\cd_l$.

\begin{prop}[{\cite[Chapter II, Section 3.3, Remark 1]{Kat78}, \cite[Lemma 7]{KK86}}]\label{lem:KatRem1}
Let $F$ be a field of $\ch(F) = p\ge 0$. 
For any prime $l\neq p$, 
the field $F$  is a $\fB_s(l)$-field if and only if $\cd_l(F)\le s$. 
\end{prop}
From the above proposition, 
when $\ch(F) = 0$, 
$F$ is $\fB_s$ if and only if $\cd(F)\le s$. 
Note that if $\ch(F) = p>0$, then $\cd_p(F) \le 1$ (\cite[Proposition 6.5.10]{NSW08}). 


\begin{prop}[{\cite[Chapter II, Section 3.3, Proposition 2]{Kat78}}]\label{prop:cdvf}
    Let $F$ be a complete discrete valuation field with residue field $k$. 
    For a prime $l$, and $s\ge 0$, 
    $F$ is a $\fB_{s+1}(l)$-field if and only if $k$ is a $\fB_s(l)$-field.
\end{prop}

\begin{lem}[{\cite[Chapter II, Section 3.3, Lemma 10]{Kat78}}]\label{lem:KatLem10}
    If $F$ is a $\fB_s(l)$-field, then it is $\fB_n(l)$ for all $n\ge s$ and 
    $K_n^M(F)$ is $l$-divisible for any $n> s$.
\end{lem}

Recall that, for $N\ge 0$, 
an \textbf{$N$-dimensional local field $F$} is a complete discrete valuation field whose residue field $k$ is an $(N-1)$-local field, and $0$-dimensional local field is a finite field. 
For an $N$-dimensional local field $F$, 
there is the tuple of its consecutive residue fields
$( F= F_N, k = F_{N-1}, F_{N-2}, \ldots, F_1, F_0)$. 
The field $F_0$ is a finite field and is called the \textbf{last residue field} of $F$. 

\begin{prop}[{\cite[Chapter II, Section~3.3, Corollary (to Proposition 2)]{Kat78}}]\label{prop:Kato}
Let $F$ be an $N$-dimensional local field. 
Then, $F$ is a $\fB_{N+1}$-field. 
For any $n \ge N+2$, $K_n^M(F)$ is divisible. 
\end{prop}

Let $F$ be a number field, that is, a finite extension of $\Q$. 
It is known that $\cd_l(F) \le 2$ if $F$ is totally imaginary, or $l\neq 2$ (\cite[{Chapter II}, Section 4.4, Proposition 13]{SerreCG}). 
By \autoref{lem:KatRem1}, 
$F$ is a $\fB_2(l)$-field for any prime $l\neq 2$. 
Moreover, when $F$ is totally imaginary, $F$ is a $\fB_2$-field. 
In the same way, 
if $F$ is a function field of one variable over a finite field of $\ch(F) = p$ 
then we have $\cd(F) = 2$ (\cite[{Chapter II}, Section 4.2, Corollary]{SerreCG}) and hence $F$ is $\fB_2(l)$ for any $l\neq p$. 
For the prime $p = \ch(F)$, 
it is known that $K_2^M(F)$ is $p$-divisible (\cite[Chapter I, Corollary 5.13]{BT73}).
Hence, $\fB_2(p)$ holds trivially.

\begin{lem}\label{lem:global_B2}
    A totally imaginary number field, as well as a function field of one variable over a finite field, 
    is a $\fB_2$-field.
\end{lem}

More generally, 
if $F$ is a finitely generated extension of a finite field $\bF_q$ 
of transcendence degree $N$, 
then $\cd_l(F) \le N + 1$ and hence $F$ is $\fB_{N+1}(l)$ for any prime $l\neq \ch(F)$.

\section{Higher Chow groups}
\label{sec:CH}

In this section we prove our main theorem (\autoref{thm:main} and \autoref{cor:main}), 
describing 
the higher Chow group $CH^{\dim(X)+s}(X, s)$ 
as the direct sum of $K_s^M(F)$ and a divisible group 
for a smooth projective variety $X$ over such a $\fB_s$-field $F$.

\subsection*{Main theorem}
To prove our main theorems (\autoref{thm:main}, \autoref{cor:main}), 
we first establish several lemmas. 

\begin{lem}
\label{lem:norm}
    Let $\cM$ be a Mackey functor over a  $\fB_s(l)$-field $F$ for some $s\ge 0$ and a prime $l$. 
    For any finite extension $F'/F$, the transfer map 
    \[
    \tr_{F'/F}\colon (\cM\otimesM \Gm^{\otimesM s})(F')/l^r \to (\cM\otimesM \Gm^{\otimesM s})(F)/l^r
    \]
    is surjective for any $r\ge 1$. 
    Here, $\Gm^{\otimesM 0}$ is defined to be the unit $\Z$ of the tensor structure in the category of Mackey functors. 
    Moreover, the transfer map for $(\cM\otimesM K_{s}^M)/l^r$ is surjective as well. 
\end{lem}
\begin{proof}
Consider the following commutative diagram with exact rows:
\[
\xymatrix{
(\cM\otimesM \Gm^{\otimesM s})(F')/l^r\ar[d] \ar[r] & (\cM\otimesM \Gm^{\otimesM s})(F')/l^{r+1}\ar[d] \ar[r] &  (\cM\otimesM \Gm^{\otimesM s})(F')/l \ar[r]\ar[d] & 0\\
(\cM\otimesM \Gm^{\otimesM s})(F)/l^r \ar[r] &  (\cM\otimesM \Gm^{\otimesM s})(F)/l^{r+1} \ar[r] &  (\cM\otimesM \Gm^{\otimesM s})(F)/l \ar[r] & 0.
}
\]
An induction on $r$, using a diagram chase in the above
commutative diagram with right exact rows, reduces the assertion
to the case $r=1$.

Recalling from \autoref{ex:MFexs} (i), 
the Mackey functor $(\cM\otimesM \Gm^{\otimesM s})/l$ is the cokernel of the morphism $l\colon \cM\otimesM \Gm^{\otimesM s} \to \cM\otimesM \Gm^{\otimesM s}$ given by the multiplication by $l$-map. 
By the right exactness of the product $-\otimesM \Gm^{\otimesM s}$, 
there is an isomorphism $(\cM\otimesM \Gm^{\otimesM s})/l \simeq \cM/l \otimesM \Gm^{\otimesM s}$. 
Since $\cM/l$ is annihilated by $l$, we have $\cM/l \otimesM \Gm^{\otimesM s} \simeq \cM/l \otimesM \Gm^{\otimesM s}/l$. 
By \cite[Theorem~4.5]{Hir24}, 
we have 
\[
(\cM\otimesM \Gm^{\otimesM s})/l \simeq \cM/l\otimesM \Gm^{\otimesM s}/l \simeq \cM/l\otimesM K_{s}^M/l.
\]
To show that the map 
$\tr_{F'/F}\colon (\cM/l\otimesM K_{s}^M/l)(F') \to (\cM/l\otimesM K_{s}^M/l)(F)$ is surjective, 
take a symbol $\set{x,\alpha}_{E/F}$ in $(\cM/l\otimesM K_{s}^M/l)(F)$ for $x\in \cM(E)/l$ and $\alpha \in K_{s}^M(E)/l$. 
We put $E' = EF'$. 
From the assumption on $F$, 
the norm $N_{E'/E} \colon K_{s}^M(E')/l \to K_{s}^M(E)/l$ is surjective (cf.~\eqref{eq:N/l^r}). 
There exists $\alpha' \in K_{s}^M(E')/l$ such that $N_{E'/E}(\alpha') = \alpha$. 
By the projection formula (cf.~\eqref{eq:pf}), we have 
\begin{align*}
      \set{x,\alpha}_{E/F} &= \set{x,N_{E'/E}(\alpha')}_{E/F} \\
      &= \set{\res_{E'/E}(x), \alpha'}_{E'/F} \\
      &= \tr_{F'/F}(\set{\res_{E'/E}(x), \alpha'}_{E'/F'}).
\end{align*}
Since $(\cM/l\otimesM K_{s}^M/l)(F)$ is generated by the symbols of the form $\set{x,\alpha}_{E/F}$ 
for some $x \in \cM(E)/l$ and $\alpha \in K_s^M(E)/l$ and an extension $E/F$, 
the above equalities imply that 
$\tr_{F'/F}\colon (\cM/l\otimesM K_{s}^M/l)(F') \to (\cM/l\otimesM K_{s}^M/l)(F)$ 
is surjective. 
\end{proof}

\begin{lem}
\label{cor:div}
    Let $G$ be a semi-abelian variety over a $\fB_s(l)$-field $F$ for some $s\ge 0$ and a prime $l$. 
    Then, $(G\otimesM \Gm^{\otimesM n})(F)$ is $l$-divisible for $n \ge s$.
\end{lem}
\begin{proof}
By the associativity of the Mackey products 
and the isomorphism $ \Gm^{\otimesM n}/l \simeq K_{n}^M/l$ (\cite[Theorem~4.5]{Hir24}), we have 
\begin{align*}
    (G\otimesM \Gm^{\otimesM n})(F)/l &\simeq ((G\otimesM \Gm^{\otimesM s}) \otimesM \Gm^{\otimesM (n-s)})(F)/l \\
    &\simeq ((G/l \otimesM K_s^M/l) \otimesM \Gm^{\otimesM (n-s)}/l)(F).
\end{align*}
Hence, it suffices to show $G/l \otimesM K_s^M/l = 0$. 
It suffices to show
$(G/l\otimesM K_s^M/l)(F)=0$.
Indeed, every finite extension $E/F$ is again a $\fB_s(l)$-field, so the same vanishing holds over $E$. 
Take any symbol $\set{x, \alpha}_{E/F}$ in $(G/l\otimesM K_s^M/l)(F)$. 
Let $[l]: G\to G$ be the multiplication by $l$ map on $G$ and 
put $E' = E([l]^{-1}x)$. 
From the assumption on $F$, 
the norm $N_{E'/E}\colon K_{s}^M(E')/l\to K_{s}^M(E)/l$ is surjective. 
There exists $\alpha' \in K_{s}^M(E')/l$ such that $N_{E'/E}(\alpha') = \alpha$. 
By the projection formula, we have 
\begin{align*}
    \set{x, \alpha}_{E/F} &= \set{x, N_{E'/E}(\alpha')}_{E/F} \\
    &= \set{\res_{E'/E}(x), \alpha'}_{E'/F} \\
    &= \set{lx', \alpha'}_{E'/F}
\end{align*}
for some $x'\in G(E')$. 
This implies the equality $(G/l \otimesM K_{s}^M/l)(F) = 0$.
\end{proof}

In the case $G = \Gm$ and $F$ is a (1-dimensional) local field, the lemma above reduces to a theorem of Sivitskii 
(\cite[Chapter~IX, Theorem~4.11]{FV02}).

\begin{lem}
\label{lem:normal}
     Let $C$ be a projective and geometrically irreducible curve over a field $F$ and 
    let $\cM$ be a Mackey functor over $F$. 
     Let $\phi\colon C'\to C$ be a finite surjective morphism which induces an isomorphism 
     \[
     C'\smallsetminus \phi^{-1}(\Sigma) \xrightarrow{\simeq} C\smallsetminus \Sigma 
     \]
     for a finite set $\Sigma \subset C$ of closed points in $C$. 
     Suppose that all closed points in $\phi^{-1}(\Sigma)$ are $F$-rational.
     Then, the map $\phi$ induces a surjective morphism 
     \[
     \Big(\underline{A}_0(C')\otimesM \cM\Big)(F) \twoheadrightarrow \Big(\underline{A}_0(C)\otimesM \cM \Big)(F).
     \]  
\end{lem}
\begin{proof}
For any finite field extension $F'/F$, put $C_{F'}  = C\otimes_F F'$ and $(C')_{F'}  = C'\otimes_FF'$. 
We denote by $\phi_{F'}\colon (C')_{F'} \to C_{F'}$ the base change of the map $\phi$ to $F'$.  
We show that 
\[
(\phi_{F'})_{\ast}\colon \underline{A}_0(C')(F') = A_0((C')_{F'}) \to 
 A_0(C_{F'}) 
\]
is surjective.

Let $\pi_{F'}\colon C_{F'} \to C$ and $\pi'_{F'}\colon (C')_{F'} \to C'$ be the projection maps. 
Put $\Sigma_{F'} := (\pi_{F'})^{-1}(\Sigma)\subset C_{F'}$. There is an isomorphism 
\[
(C')_{F'}\smallsetminus (\phi_{F'})^{-1}(\Sigma_{F'}) \xrightarrow{\simeq} C_{F'} \smallsetminus \Sigma_{F'}.
\]
Consider the following commutative diagram: 
\begin{equation}\label{eq:CH0C'}
    \vcenter{\xymatrix{
    \displaystyle \bigoplus_{x' \in (C')_{F'}\smallsetminus (\phi_{F'})^{-1}(\Sigma_{F'})}\Z \oplus \bigoplus_{x' \in (\phi_{F'})^{-1}(\Sigma_{F'})} \Z\ar@{->>}[d]\ar@{->>}[r] & CH_0( (C')_{F'}) \ar[d]^{(\phi_{F'})_{\ast}} \\
    \displaystyle\bigoplus_{x \in C_{F'}\smallsetminus \Sigma_{F'}}\Z\oplus \bigoplus_{x \in \Sigma_{F'}}\Z \ar@{->>}[r] & CH_0( C_{F'}).
}}
\end{equation}
Recall that the push forward map is given by 
$(\phi_{F'})_{\ast}([x']) = [F'(x'):F'(x)][x]$ for  $x' \in C'_{F'}$ with $x = \phi_{F'}(x')$. 
By the assumption that the closed points in $\phi^{-1}(\Sigma)$ are $F$-rational, 
the closed points in $(\phi_{F'})^{-1}(\Sigma_{F'})$ are all $F'$-rational. 
For each $x \in \Sigma_{F'}$, the fiber $(\phi_{F'})^{-1}(x)$ is finite and $F'$-rational, 
the $x$-component of the left vertical map of \eqref{eq:CH0C'} 
\[
\bigoplus_{x'\in (\phi_{F'})^{-1}(x)} \Z \to \Z 
\]
is just the summation. 
Therefore, the left vertical map is surjective in the diagram \eqref{eq:CH0C'}, and so is the right.  
Since the push-forward map 
$(\phi_{F'})_{\ast}\colon CH_0((C')_{F'})\to CH_0(C_{F'})$ 
is surjective and is compatible with the degree maps, its restriction to the degree-zero parts is also surjective. Hence
$(\phi_{F'})_{\ast}\colon A_0((C')_{F'})\to A_0(C_{F'})$ is surjective. 
This gives rise to a surjection 
\begin{equation}\label{eq:C'}
  (\underline{A}_0(C') \otimesM \cM)(F)\twoheadrightarrow
  (\underline{A}_0(C) \otimesM \cM)(F).
\end{equation}
\end{proof}

\begin{prop}
\label{prop:Van}
    Let $F$ be an infinite $\fB_s(l)$-field for some $s\ge 0$ and a prime $l$, and  
    let $X$ be an integral projective and geometrically irreducible scheme over $F$. 
    Then, 
    \[
    \Big(\underline{A}_0(X)\otimesM \Gm^{\otimesM n}\Big)(F)
    \] 
    is $l$-divisible for any $n \ge s$.
\end{prop}
\begin{proof}
For the prime $l$, 
we show $\Big(\underline{A}_0(X)\otimesM \Gm^{\otimesM n}\Big)(F)/l = 0$. 

\smallskip
\noindent
(\textbf{Reduction to a curve}) 
Assume $\dim(X) \ge 2$.
Take an element 
\[
\{\gamma,\alpha\}_{F'/F}\in \Big(\underline{A}_0(X)\otimesM \Gm^{\otimesM n}\Big)(F)
\]
with $\gamma\in A_0(X_{F'})$ and $\alpha\in \Gm^{\otimesM n}(F')$.
Since $\tr_{F'/F}(\{\gamma,\alpha\}_{F'/F'})=\{\gamma,\alpha\}_{F'/F}$, in order to show
$\Big(\underline{A}_0(X)\otimesM \Gm^{\otimesM n}\Big)(F)/l = 0$,
it is enough to consider the case $F'=F$ and to show $\set{\gamma, \alpha}_{F/F} = 0 \mod l$ 
for $\gamma \in A_0(X)$ and $\alpha \in \Gm^{\otimesM n} (F)$. 

Write $\gamma = \sum_{i=1}^ra_i [x_i] $ in $A_0(X)$ 
for $x_1,\ldots, x_r \in X_0$. 
By \cite[Theorem~1]{AK79}, 
there exists an integral geometrically irreducible curve $C \subset X$ 
such that  $\set{x_1,\ldots, x_r} \subset C$. 
The $0$-cycle $\gamma_C = \sum_{i=1}^r a_i [x_i]$ in $A_0(C)$ maps 
to $\gamma$ by the natural homomorphism $i_{\ast}\colon A_0(C) \to A_0(X)$ 
induced from the inclusion $i\colon C\hookrightarrow X$. 
In particular, $i_{\ast}$ induces a map 
\[
 \Big(\underline{A}_0(C)\otimesM \Gm^{\otimesM n}\Big)(F)/l  \to \Big(\underline{A}_0(X)\otimesM \Gm^{\otimesM n}\Big)(F)/l
\]
and this maps 
the symbol $\set{\gamma_C,\alpha}_{F/F}$ to the symbol $\set{\gamma, \alpha}_{F/F}$. 
Therefore, it remains to show $\Big(\underline{A}_0(C)\otimesM \Gm^{\otimesM n}\Big)(F)/l  = 0$.

\smallskip
\noindent
(\textbf{Reduction to a smooth curve})
Let $X=C$ be an integral projective and geometrically irreducible curve
over $F$. By
\cite[\href{https://stacks.math.columbia.edu/tag/0BY4}{Lemma 0BY4}]
{stacks-project}, there exists a finite purely inseparable extension
$F'/F$ such that, if
$D := (C_{F'})_{\red}$ and $\phi\colon C' := D^N\to D$ 
is the normalization of $D$,
then $C'$ is smooth and projective over $F'$.

Since nilpotent thickenings do not affect zero-cycles, the closed
immersion $D\hookrightarrow C_{F'}$ induces an isomorphism of Mackey functors
$\underline{A}_0(D) \xrightarrow{\simeq} \underline{A}_0(C_{F'})$.
$\phi$ is finite and is an
isomorphism outside a finite set $\Sigma\subset D$.

Choose a finite extension $F''/F'$ such that, for the base change
$\phi_{F''}\colon C'_{F''}\to D_{F''}$,
all closed points in $\phi_{F''}^{-1}(\Sigma_{F''})$ 
are
$F''$-rational, where $\Sigma_{F''}:=\Sigma\otimes_{F'}F''$.
Such an extension exists because $\phi^{-1}(\Sigma)$ is finite over $F'$. 
Therefore \autoref{lem:normal} applies to
$\phi_{F''}\colon C'_{F''}\to D_{F''}$, and gives a surjection
\[
  \bigl(\underline{A}_0(C'_{F''})\otimesM
  \Gm^{\otimesM n}\bigr)(F'')/l
  \twoheadrightarrow
  \bigl(\underline{A}_0(D_{F''})\otimesM
  \Gm^{\otimesM n}\bigr)(F'')/l .
\]
Using the nilpotent invariance above, we identify
$\underline{A}_0(D_{F''})
  \simeq
  \underline{A}_0(C_{F''})$.
Hence we get a surjection
\[
  \bigl(\underline{A}_0(C'_{F''})\otimesM
  \Gm^{\otimesM n}\bigr)(F'')/l
  \twoheadrightarrow
  \bigl(\underline{A}_0(C_{F''})\otimesM
  \Gm^{\otimesM n}\bigr)(F'')/l .
\]

On the other hand, since $F$ is a $\fB_s(l)$-field, \autoref{lem:norm} gives a surjection
\[
  \tr_{F''/F}\colon
  \bigl(\underline{A}_0(C_{F''})\otimesM
  \Gm^{\otimesM n}\bigr)(F'')/l
  \twoheadrightarrow
  \bigl(\underline{A}_0(C)\otimesM
  \Gm^{\otimesM n}\bigr)(F)/l .
\]
Therefore the composite
\[
  \bigl(\underline{A}_0(C'_{F''})\otimesM
  \Gm^{\otimesM n}\bigr)(F'')/l
  \twoheadrightarrow
  \bigl(\underline{A}_0(C)\otimesM
  \Gm^{\otimesM n}\bigr)(F)/l
\]
is surjective.

Thus, in order to prove $\bigl(\underline{A}_0(C)\otimesM
  \Gm^{\otimesM n}\bigr)(F)/l=0$,
it is enough to prove the same assertion for the smooth projective curve
$C'_{F''}$ over $F''$. 
Replacing $F$ by $F''$ and $C$ by
$C'_{F''}$, we may assume that $C$ is smooth, projective and
geometrically irreducible over $F$. 

Finally, by \autoref{lem:norm}, after replacing $F$ by a further finite extension if necessary,
we may also assume that $C(F)\neq\emptyset$.

\smallskip
\noindent
(\textbf{Proof of the $l$-divisibility})
Let $J = \Jac_{C}$ be the Jacobian variety of $C$. 
For any finite extension $F'/F$, the Abel--Jacobi map 
\[
A_0(C_{F'})\xrightarrow{\simeq} J(F') \simeq \Jac_{C_{F'}}(F') 
\]
is an isomorphism (cf.~\cite[Chapter 7, Theorem 4.39]{Liu02}) 
and this gives rise to an isomorphism 
$\underline{A}_0(C) \xrightarrow{\simeq} J$ 
of Mackey functors. 
By \autoref{cor:div}, 
we obtain 
$(J\otimesM \Gm^{\otimesM n})(F)/l = 0$.  
This proves the proposition. 
\end{proof}

Recall that the index $\delta(X/F)$ of $X$ over $F$ is the greatest common divisor of the degrees of all closed points of $X$.
In particular, if $X$ has a $F$-rational point, then $\delta(X/F)=1$.
If $F$ is a finite field, then $X$ has a zero-cycle of degree one, i.e.~$\delta(X/F)=1$ (\cite[1.5.3, Lemma~1]{Sou}).

\begin{thm}\label{thm:main}
    Let $F$ be a $\fB_s(l)$-field for some $s\ge 0$ and a prime $l$ and 
    let $X$ be a smooth projective and geometrically irreducible scheme over $F$ of $d = \dim(X)$.
    Then, $CH^{d+n}(X,n)$ is $l$-divisible for any $n\ge s+1$.
	If $l\nmid \delta(X/F)$, then 
    the map $f_\ast^{CH}$ induces an isomorphism 
    \[
        CH^{d+s}(X,s)/l^r \xrightarrow{\simeq} K_{s}^M(F)/l^r
    \]
    for any $r \ge 1$. 
\end{thm}

\begin{proof}
When the field $F$ is finite, $F$ is a $\fB_1$-field. 
The theorem in this setting follows from \autoref{thm:Akh}.
Hence, we may assume that $F$ is an infinite $\fB_s(l)$-field for a prime $l$. 
For $n\ge s+1$, 
by \eqref{eq:MK-Chow}, there is a surjective homomorphism
\[
    \bigoplus_{x\in X_0}K_{n}^M(F(x))/l \twoheadrightarrow CH^{d+n}(X,n)/l. 
\]
By \autoref{lem:KatLem10}, 
we have $K_{n}^M(F(x))/l = 0$. The assertion $CH^{d+n}(X,n)/l = 0$ follows from this.

Let $\cM$ be the quotient $\underline{CH}_0(X)/\underline{A}_0(X)$ as a Mackey functor.
From the following commutative diagram of Mackey functors with exact rows
\begin{equation}
\xymatrix{
 0\ar[r]&\cM\ar[d]^{l^r} \ar[r] &\Z \ar[d]^{l^r}\ar[r] & \Z/\cM \ar[d]^{l^r} \ar[r] & 0\\
 0\ar[r]&\cM \ar[r] &\Z \ar[r] & \Z/\cM  \ar[r] & 0,
 }
\end{equation}
we obtain an exact sequence of Mackey functors
\[
0\to (\Z/\cM)[l^r]\to \cM/l^r\to \Z/l^r\to (\Z/\cM)/l^r\to 0.
\]
A zero-cycle on $X$ of degree $\delta(X/F)$ remains a
zero-cycle of the same degree after every finite extension of $F$.
Hence $\delta(X/F)\Z\subset\cM$ as Mackey functors 
and $\Z/\cM$ is $\delta(X/F)$-torsion.
By the assumption $l\nmid \delta(X/F)$, we have $(\Z/\cM)[l^r]=(\Z/\cM)/l^r=0$.
Hence we have
$$
\cM/l^r\simeq  \Z/l^r
$$
which shows 
\[
(\cM\otimesM \Gm^{\otimesM s})(F)/l^r \simeq (\cM/l^r\otimesM \Gm^{\otimesM s})(F)\simeq (\Z/l^r\otimesM \Gm^{\otimesM s})(F)\simeq K^M_s(F)/l^r.
\]
Here, the last isomorphism is given by \cite[Theorem 4.5]{Hir24}.
This induces a commutative diagram with exact rows: 
\begin{equation}\label{diag:main}
\vcenter{
\xymatrix@C=4mm{
 (\underline{A}_0(X)\otimesM \Gm^{\otimesM s})(F)/l^r \ar[r] &(\underline{CH}_0(X)\otimesM \Gm^{\otimesM s})(F)/l^r \ar@{->>}[d]\ar[r] & (\cM \otimesM \Gm^{\otimesM s})(F)/l^r \ar[d]^{\simeq} \ar[r] & 0\\
 & CH^{d+s}(X,s)/l^r \ar[r]^-{f_{\ast}^{CH}} & K_s^M(F)/l^r.
 }}
 \end{equation}
From \autoref{thm:Akh2} the middle vertical map in the above diagram is surjective.
For the prime $l$, we have $(\underline{A}_0(X)\otimesM \Gm^{\otimesM s})(F)/l^r = 0$ by \autoref{prop:Van}. 
From the above diagram, 
$f_{\ast}^{CH}\colon CH^{d+s}(X,s)/l^r \to K_s^M(F)/l^r$ is an isomorphism for any $r\geq 1$.
\end{proof}

\begin{cor}\label{cor:main}
    Let $F$ be a $\fB_s$-field for some $s\ge 0$ 
    and let $X$ be an integral projective and geometrically irreducible scheme over 
    $F$ of $d = \dim(X)$. 
    Then, 
    \[
    f_{\ast}^{CH}\colon CH^{d+s}(X,s) \twoheadrightarrow K_s^M(F)
    \]
    is surjective. 
    If we further assume that $X$ is smooth over $F$ 
    and $\delta(X/F)=1$
    then 
    \[
    CH^{d+s}(X,s) \simeq K_s^M(F) \oplus A^{d+s}(X,s)
    \]
    with divisible group $A^{d+s}(X,s)$.
\end{cor}
\begin{proof}
    There exists a finite field extension $F'/F$ such that $X(F')\neq \emptyset$. 
    The map $f_{\ast}^{CH}:CH^{d+s}(X_{F'},s) \to K_s^M(F')$ for $X_{F'} = X\otimes_F F'$ 
    is surjective. Consider the following commutative diagram: 
    \begin{equation}
    \xymatrix{
        CH^{d+s}(X_{F'}, s) \ar@{->>}[r]^-{f_{\ast}^{CH}}\ar[d]  & K_s^M(F') \ar@{->>}[d]^{N_{F'/F}}   \\
        CH^{d+s}(X,s)\ar[r]^-{f_{\ast}^{CH}} & K_s^M(F).
        }
    \end{equation}
    Since $F$ is a $\fB_s$-field, the right vertical map in the above diagram is surjective. 
    From the above diagram, $f_{\ast}^{CH}\colon  CH^{d+s}(X,s) \to K_s^M(F)$ is surjective. 
    
    Now, we further assume that $X$ is smooth over $F$. 
    Suppose that $F$ is finite. 
    If $s=1$, the assertion
    follows from \autoref{thm:Akh}. If \(s\ge2\), then
    $CH^{d+s}(X,s)=0$ by \autoref{thm:Akh}, while
    $K_s^M(F)=0$. Thus the assertion also holds in this case.
    Suppose $F$ is an infinite $\fB_s$-field. 
    Using Akhtar's $K$-group (cf.~\autoref{thm:Akh2}), we have 
    \begin{equation}
    	K(F;\underline{CH}_0(X),\overbrace{\Gm,\ldots,\Gm}^s) \simeq CH^{d+s}(X,s).
    \end{equation}
    There is a surjective homomorphism 
    \[
    	\psi\colon (\underline{CH}_0(X)\otimesM K_{s}^M)(F)\to K(F;\underline{CH}_0(X),\overbrace{\Gm,\ldots,\Gm}^s)
    \]
    defined by 
    \[
    	\psi(\set{\alpha, \set{x_1,\ldots, x_s}_{F'}}_{F'/F}) = \set{\alpha, x_1,\ldots, x_s}_{F'/F} 
    \]
    for a finite extension $F'/F$ and a symbol $\set{x_1,\ldots,x_s}_{F'}$ in $K_s^M(F')$. 
	It is well-defined and surjective by the arguments in \cite[Theorem 1.4]{Som90}. 
    By the assumption $\delta(X/F)=1$,
    we have the short exact sequence of Mackey functors 
    $$0\to \underline{A}_0(X)\to \underline{CH}_0(X)\to \Z \to 0$$
    which induces 
    the following commutative diagram with exact rows: 
	\[
	\xymatrix{
	& (\underline{A}_0(X)\otimesM K_s^M)(F)\ar[d]\ar[r] & 	(\underline{CH}_0(X)\otimesM K_s^M)(F) \ar[r] \ar@{->>}[d]^{\psi}& (\Z\otimesM K_s^M)(F) \ar[d]^{\simeq} \ar[r] & 0  \\
	0 \ar[r] & A^{d+s}(X,s)\ar[r] & CH^{d+s}(X,s)\ar[r]^-{f_\ast^{CH}} & K_s^M(F) \ar[r] & 0.
	}
	\]
	From this diagram we obtain that the left vertical map is surjective. 
    Since there is a surjective map 
    \[
    (\underline{A}_0(X)\otimesM \Gm^{\otimesM s})(F)\twoheadrightarrow  (\underline{A}_0(X)\otimesM K_s^M)(F), 
    \]
    the kernel $A^{d+s}(X,s)$ of the map $f_{\ast}^{CH}$ is divisible by \autoref{prop:Van}.
    Note that a divisible group is an injective object in the category of abelian groups 
     (\cite[\href{https://stacks.math.columbia.edu/tag/01D6}{Section 01D6}]{stacks-project}), 
    the short exact sequence 
    \[
    0 \to A^{d+s}(X,s) \to CH^{d+s}(X,s) \to K_s^M(F) \to 0
    \]
    splits (\cite[\href{https://stacks.math.columbia.edu/tag/0136}{Lemma 0136}]{stacks-project}). 
\end{proof}

\begin{rem}
\begin{enumerate}
\item 
When $F$ is algebraically closed, it is a $\fB_0$-field. 
As we have $K_0^M(F)=\Z$, 
we have $CH_0(X) \simeq \Z \oplus A_0(X)$. 
The $l$-primary torsion part $A_0(X)\{l\}$ is isomorphic to the Albanese variety $\mathrm{Alb}_X(F)\{l\}$ which is $l$-divisible for any prime $l$ (\cite{Roj80}, \cite{Mil82}). 
Our result implies $A_0(X)$ is divisible although it may not be isomorphic to an abelian variety. 

\item 
Let $p$ be an odd prime and, 
consider $X = \Spec(\Qp(\zeta_p))$ over the $p$-adic field $\Qp$  with a primitive $p$-th root of unity $\zeta_p$. 
The scheme $X$ is not geometrically irreducible and $\Qp$ is a $\fB_2$-field. 
In this case, the kernel $A^2(X,2)$ is not divisible. 
In fact, 
$f_\ast^{CH} \colon CH^{0+2}(X,2)\simeq K_2^M(\Qp(\zeta_p)) \to K_2^M(\Qp)$ is just the norm map. 
From the structure theorem (\cite[Chapter IX, Theorems 4.3 and 4.7]{FV02}), 
we have 
\[
K_2^M(\Qp(\zeta_p)) \simeq \Z/p(p-1) \oplus D',\ \mbox{and}\ K_2^M(\Qp) \simeq \Z/(p-1) \oplus D,
\]
where $D$ and $D'$ are uniquely divisible groups. 
Since $K_2^M(\Qp)$ is $p$-torsion free, 
there is an exact sequence 
\[
0 \to A^{2}(X,2)/p \to K_2^M(\Qp(\zeta_p))/p \to K_2^M(\Qp)/p \to 0.
\]
This implies that $A^{2}(X,2)/p \simeq K_2^M(\Qp(\zeta_p))/p \simeq \Z/p$. 
\end{enumerate}
\end{rem}

\subsection*{Applications}
In what follows, we describe the structure of $CH^{d+s}(X,s)$ using our main result (\autoref{thm:main}) for arithmetic fields whose Milnor $K$-groups are explicitly known. 
First, we consider a (1-dimensional) local field.
As noted in \autoref{sec:intro}, the corollary below refines the results in \cite[Theorem 1.7 (3), Corollary 1.8]{GKR22}.

\begin{cor}\label{cor:CH}
    Let $X$ be a smooth projective and geometrically irreducible scheme over a local field $F$ of $d = \dim(X)$. 
    Then $CH^{d+n}(X,n)$ is divisible for $n>2$.  
    If we further assume $\delta(X/F)=1$, then there is an isomorphism
    \[
    CH^{d+2}(X,2)\simeq \mu(F)\oplus mK_2^M(F)\oplus A^{d+2}(X, 2).
    \]
    Here, $\mu(F)$ denotes the group of roots of unity in $F$,  $m:=\#\mu(F)$, $mK_2^M(F)$ is an uncountable, uniquely divisible group and $A^{d+2}(X,2)$ is divisible.
\end{cor}
\begin{proof}
Note that the local field $F$ is a $\fB_2$-field (\autoref{prop:Kato}). 
Hence, $CH^{d+n}(X,n)$ is divisible for $n>2$ by \autoref{thm:main}. 
By Moore's theorem, we have 
$K_2^M(F) \simeq \mu(F) \oplus m K_2^M(F)$, 
where $mK_2^M(F)$ is uncountable and uniquely divisible and $m = \# \mu(F)$ (\cite[Chapter IX, Theorems 4.3 and 4.7]{FV02}). 
From \autoref{cor:main}, there is a decomposition 
\[
 CH^{d+2}(X,2) \simeq K_2^M(F) \oplus A^{d+2}(X,2).
\]
This proves the claim.
\end{proof}

\autoref{cor:CH} can be generalized to $N$-dimensional local fields as follows.

\begin{cor}\label{cor:CH_high}
    Let $F$ be an $N$-dimensional local field and 
    let $X$ be a smooth projective and geometrically irreducible scheme over $F$ of $d = \dim(X)$. 
    Then, $CH^{d+n}(X,n)$ is divisible for $n>N+1$.
    If we further assume $\delta(X/F)=1$,
    then there is an isomorphism 
    \[
    CH^{d+N+1}(X,N+1) \simeq \mu(F)\oplus mK_{N+1}^M(F)\oplus A^{d+N+1}(X, N+1).
    \]
    Here, $m:=\#\mu(F)$, $mK_{N+1}^M(F)$ is a divisible group and $A^{d+N+1}(X,N+1)$ is divisible.
    \end{cor}
\begin{proof}
    By \autoref{prop:Kato}, $F$ is a $\fB_{N+1}$-field. 
    From \autoref{thm:main} and \autoref{cor:main}, 
    $CH^{d+n}(X,n)$ is divisible for any $n> N+1$, and 
    the structure morphism $f\colon X\to \Spec(F)$ induces a surjective homomorphism  
    \[
      f_\ast^{CH}\colon CH^{d+N+1}(X,N+1) \to K_{N+1}^M(F)
    \]
    whose kernel is divisible.
    Therefore, 
    $CH^{d+N+1}(X,N+1)$ decomposes as the direct sum of a divisible group $A^{d+N+1}(X,N+1)$ and  $K_{N+1}^M(F)$ (\cite[\href{https://stacks.math.columbia.edu/tag/01D6}{Section 01D6}]{stacks-project}). 
    Now, the claim follows from the following proposition.
\end{proof}

\begin{prop}
    Let $F$ be an $N$-dimensional local field and let $\mu(F)$ be the group of roots of unity in $F$. Put $m = \# \mu(F)$.
    Then we have
    \[
    K_{N+1}^M(F)\simeq \mu(F)\oplus mK_{N+1}^M(F)
    \]
    where $mK_{N+1}^M(F)$ is divisible.
\end{prop}
\begin{proof}
Let $\bF_q$ be the last residue field of $F$ of characteristic $p>0$.
Suppose first that $\ch(F)>0$. Then $\mu(F)=\bF_q^\times$ and $m=q-1$.
Let $K_{N+1}^{\mathrm{top}}(F)$ be the topological Milnor $K$-group of $F$ (cf.~\cite[Section 6.3]{Fes00a}).
The kernel of the natural quotient map $K_{N+1}^M(F) \twoheadrightarrow K_{N+1}^{\mathrm{top}}(F)$ 
is divisible (\cite[Section 6.6, Theorem 2]{Fes00a}). 
Hence, we have $K_{N+1}^M(F) = D \oplus K_{N+1}^{\mathrm{top}}(F)$ for a divisible group $D$. 
It is also known that $K_{N+1}^{\mathrm{top}}(F) \simeq \bF_q^{\times}$, and thus $(q-1)K_{N+1}^M(F)=D$ is divisible.  

In the case $\ch(F)=0$, we have $m=p^r(q-1)$ for some $r\geq 0$.
By the norm residue isomorphism theorem (the Bloch-Kato conjecture) and \cite[Chapter II, Section 1.1, Theorem 3]{Kat80}, we have
\begin{align*}
K_{N+1}^M(F)/m 
&\simeq H^{N+1}(F,\mu_m^{\otimes (N+1)})\\
&\simeq H^{N+1}(F,\mu_m^{\otimes N})(1)\\
&\simeq \Z/m(1)= \mu(F).
\end{align*}
Let $\zeta_m\in F$ be a primitive $m$-th root of unity.
We now show that an $m$-torsion element $\{\alpha,\zeta_m\}\in K^M_{N+1}(F)$ for some $\alpha\in K^M_N(F)$ is a generator of $K^M_{N+1}(F)/m$.
It is clear that an element of the form $\{\alpha, \zeta_m\}$ is $m$-torsion.
We will show that there exists an element $\alpha \in K^M_N(F)$ such that $\{\alpha, \zeta_m\}$ is a generator of $K^M_{N+1}(F)/m$.
We consider the Hilbert symbol map
\[
    K^M_N(F)/m\times F^\times/m \to \mu_m;\ (\alpha, \beta)\mapsto (\alpha, \beta)_m := \gamma^{\rho(\alpha)-1},\ \gamma^m=\beta,
\]
where $\rho\colon K^M_N(F)\to \Gal(F^{\ab}/F)$ is the reciprocity map for higher local class field theory (\cite{Kat80}).
For $\beta=\zeta_m$, we take $\gamma=\zeta_{m^2}$, a primitive $m^2$-th root of unity.
Then for a cyclic extension $F(\zeta_{m^2})/F$ of degree $m$, the reciprocity map $\rho$ induces an isomorphism
\[
    K^M_N(F)/N_{F(\zeta_{m^2})/F}K^M_N(F(\zeta_{m^2}))\simeq \Gal(F(\zeta_{m^2})/F)=\langle \sigma \rangle.
\]
Here $\sigma\in \Gal(F(\zeta_{m^2})/F)$ is a generator mapping $\zeta_{m^2}$ to $\zeta_{m^2}^{1+m}$.
Thus there exists an element $\alpha\in K^M_N(F)$ such that $\rho(\alpha)=\sigma$.
Then we have $(\alpha,\zeta_m)_m=\zeta_m$ and thus the element $\{\alpha, \zeta_m\}$ is a generator of $K^M_{N+1}(F)/m$.

For a prime $l$ dividing $m$, let $e$ denote the $l$-adic valuation of $m$.
Let $e':=m/l^e$.
Then we have an exact sequence
\[
K_{N+1}^M(F)/l^{e}\xrightarrow{l} K_{N+1}^M(F)/l^{e+1}\to K_{N+1}^M(F)/l\to 0.
\]
Since the element $\{\alpha, \zeta_m\}$ is a generator of $K^M_{N+1}(F)/m$, an element $\{\alpha, \zeta_m^{e'}\}$ is a generator of $K^M_{N+1}(F)/l^e$ and hence the order of the image of $K_{N+1}^M(F)/l^{e}\xrightarrow{l} K_{N+1}^M(F)/l^{e+1}$ is less than or equal to $l^{e-1}$.
Thus we have $\# K_{N+1}^M(F)/l^{e+1} \le l^{e-1}\cdot l=l^{e}$, since $K_{N+1}^M(F)/l\simeq \mu_{l}$.
Therefore the canonical surjection $K_{N+1}^M(F)/l^{e+1}\to K_{N+1}^M(F)/l^{e}$ is an isomorphism, and hence $l^e K_{N+1}^M(F)$ is $l$-divisible.

For any prime $l$ with $(l, m)=1$, take $E=F(\mu_l)$.
Since the composite map 
\[
N_{E/F}\circ \res_{E/F}\colon K_{N+1}^M(F)/l\to K_{N+1}^M(E)/l\to K_{N+1}^M(F)/l 
\]
is the multiplication by $[E:F]$ and $[E:F]$ is coprime to $l$, $\res_{E/F}: K_{N+1}^M(F)/l\to K_{N+1}^M(E)/l$ is injective.
Moreover, $K_{N+1}^M(F)/l$ is a subgroup of $(K_{N+1}^M(E)/l)^{\Gal(E/F)}$.
On the other hand, we have
\[
(K_{N+1}^M(E)/l)^{\Gal(E/F)}\simeq H^{N+1}(E,\mu_l^{\otimes (N+1)})^{\Gal(E/F)}\simeq \mu_l(E)^{\Gal(E/F)}=1.
\] 
Hence $K_{N+1}^M(F)/l$ is trivial and thus $K_{N+1}^M(F)$ is $l$-divisible.
This completes the proof.
\end{proof}

Comparing \autoref{prop:Akh04_7.2}, 
for a global field $F$, 
we obtain the following more precise results 
on the torsion group $CH^{d+2}(X,2)$.

\begin{cor}\label{cor:CH_FF}
    Let $F$ be a global field 
    and $X$ a smooth projective and geometrically irreducible scheme over $F$ of $d = \dim(X)$. 
    Assume 
    $\delta(X/F)=1$
    and that $F$ is totally imaginary if $\ch(F) = 0$. 
    Then, 
    \[
    CH^{d+2}(X,2) \simeq K_2^M(F) \oplus A^{d+2}(X,2)
    \]
    where $A^{d+2}(X,2)$ is divisible, 
    and $K_2^M(F)$ is torsion. 
    Moreover, when $\ch(F) = p>0$, $K_2^M(F)$ is prime to $p$-torsion.
\end{cor}
\begin{proof}
The field $F$ is a $\fB_2$-field by \autoref{lem:global_B2}.
By \autoref{cor:main}, 
$CH^{d+2}(X,2)$ is the direct sum of $K_2^M(F)$ and the divisible group $A^{d+2}(X,2)$. 

If $\ch(F) = 0$ (and $F$ is totally imaginary in our setting), then 
the tame symbol map $\partial$ gives a short exact sequence 
\[
0\to K_2(\cO_F) \to K_2^M(F) \xrightarrow{\partial} \bigoplus_v \bF_v^{\times} \to 0,
\]
where $v$ runs through the set of finite places of $F$ 
and $\bF_v$ is the residue field of $F$ at $v$ (cf.~\cite[Chapter V, Section 6, Theorem 6.8]{Wei13}). 
Here, the kernel of the map $\partial$ coincides with the Quillen $K$-group $K_2(\cO_F)$ of the ring of integers 
$\cO_F$ of $F$. 
By a theorem of Garland, $K_2(\cO_F)$ is finite and 
$K_2(F) = K_2^M(F)$ is torsion. 

If $\ch(F) = p>0$,  
    let $C$ be the smooth projective curve over $\bF_q$ whose function field is $F = \bF_q(C)$. 
The localization sequence of the Quillen $K$-theory gives 
\[
0 \to K_2(C) \to K_2(F) \xrightarrow{\partial} \bigoplus_{x\in C_0}\bF_q(x)^{\times} \to K_1(C) \to \bF_q^\times \to 0.
\]
The cokernel of $\partial$ is $\bF_q^\times$  and the kernel $K_2(C)$ of $\partial$ 
is finite of order prime to $p$ (\cite[Chapter II, Section 2]{BT73}, see also \cite[Section 5.5]{Wei05}).
In particular, $K_2^M(F) = K_2(F)$ is torsion of order prime to $p$. 
\end{proof}

For example, 
when $X$ is defined over $\Q$ with $\delta(X/\Q) =1$, 
the tame kernel $K_2(\Z)$ is isomorphic to $\Z/2$, 
and hence we obtain 
$CH^{d+2}(X,2)/l \simeq \bigoplus_p \bF_p^{\times}/l$ 
for any prime $l>2$. 

Finally, we conclude this section by stating the result for the case where $F=\R$.
The real field $\R$ is a $\fB_0(l)$-field for any $l\neq 2$ and it is not a $\fB_0(2)$-field.
In fact, $K^M_n(\R)/2\simeq \Z/2$  (\cite[Chapter III, Example 7.2 (c)]{Wei13}). 
The lemma below implies that $A^{d+n}(X,n)/2$ may become a non-trivial finite group, 
namely, the map $f_{\ast}^{CH}\colon CH^{d+n}(X,n)/2 \to K_n^M(\R)/2$ may not be injective. 
Moreover, 
one can show that the map $f_{\ast}^{CH}$ modulo 2 is not surjective if $X(\R) = \emptyset$. 

\begin{lem}\label{lem:R}
    Let $X$ be a smooth projective and geometrically irreducible scheme over $\R$ of dimension $d$. 
    \[
    CH^{d+n}(X,n)/2 \simeq (\Z/2)^{\# \pi_0(X(\R))}
    \]
    for any $n\ge 1$, where $\pi_0(X(\R))$ denotes the set of connected components of $X(\R)$.
\end{lem}
\begin{proof}
First, we suppose $X(\R) = \emptyset$. 
In this case, 
$\bigoplus_{x\in X_0}K_n^M(\R(x))/2 = 0$ because $K_n^M(\C)$ is divisible. 
By \eqref{eq:MK-Chow}, $CH^{d+n}(X,n)$ is $2$-divisible. 

By contrast, 
in the case where $X(\R)\neq \emptyset$, 
we have decompositions 
$CH_0(X)\simeq \Z\oplus A_0(X)$ 
and $CH^{d+n}(X,n) \simeq K_n^M(\R) \oplus A^{d+n}(X,n)$  
which induce the following isomorphism by \autoref{thm:Akh2};
\[
K(\R; \underline{A}_0(X),  \overbrace{\Gm,\ldots, \Gm}^n)/2\simeq A^{d+n}(X,n)/2. 
\]
Since $\Z$ is the unit object with respect to Mackey product, we have
\[
(\underline{A}_0(X)\otimesM\Z/2)(\R)=A_0(X)/2=K(\R;\underline{A}_0(X),\Z)/2.
\]
We also have 
\[
\Phi\colon (\underline{A}_0(X)\otimesM \Gm^{\otimesM n}/2)(\R)\xrightarrow{\simeq} A_0(X)/2,
\] 
since $(\underline{A}_0(X)\otimesM \Gm^{\otimesM n}/2)(\C)=0$.

It remains to show that the relations of the Weil reciprocity law 
in $K(\R;\underline{A}_0(X),\overbrace{\Gm,\ldots,\Gm}^n)$ vanish under this isomorphism $\Phi$.
Let $K=\R(C)$ be the function field of a smooth projective curve $C$ over $\R$. 
After grouping the multiplicative factors, a relation of the Weil reciprocity law 
is represented modulo $2$ by 
\begin{equation}\label{eq:WR}
\sum_{x\in C_0} \set{s_x(\Gamma),\partial_x(\alpha)}_{\R(x)/\R}, 
\end{equation}
where $\Gamma\in A_0(X_K), \alpha\in K_{n+1}^M(K)/2$, 
$s_x\colon CH_0(X_{\R(C)})\to CH_0(X_x)$ is the specialization map, 
and 
$\partial_x\colon K_{n+1}^M(K)/2\to K_n^M(\R(x))/2$ is the boundary map. 
Here, only finitely many terms are non-zero.
Since $K(\sqrt{-1})$ is a function field of one variable over $\C$, we have $\cd_2(K(\sqrt{-1}))=1.$ 
The norm residue isomorphism theorem and the exact sequence associated with $K(\sqrt{-1})/K$ 
(\cite[Theorem 2.1]{OVV07})
imply that multiplication by $\set{-1}$ is surjective:
\[
K_q^M(K)/2 \xrightarrow{\set{-1}\cdot}K_{q+1}^M(K)/2\qquad (q\ge 1).
\]
Iterating this surjectivity, there exists $a\in K^\times$ such that 
$\alpha=\set{a,-1,\ldots,-1}$. 
Therefore, $\partial_x(\alpha) = \mathrm{ord}_x(a)\set{-1,\ldots,-1}$. 
By the projection formula, the above relation \eqref{eq:WR} is equal to 
\[
\sum_{x\in C_0} \mathrm{ord}_x(a) \set{ N_{\R(x)/\R}\bigl(s_x(\Gamma)\bigr), -1,\ldots,-1 }_{\R/\R} =0, 
\]
since the specialization of zero-cycles gives $\sum_{x\in C_0} \mathrm{ord}_x(a) N_{\R(x)/\R}\bigl(s_x(\Gamma)\bigr) =0$ in $A_0(X)$.

Finally, 
the real cycle class map gives 
$\mathrm{cl}_{\R}\colon CH_0(X)\to H_0(X(\R),\Z/2)$, 
and we have $CH_0(X)/2 \simeq (\Z/2)^{\#\pi_0(X(\R))}$ 
(\cite[Proposition~3.2]{CI81}). Hence, we have 
\[
A^{d+n}(X,n)/2 \simeq K(\R;\underline{A}_0(X),\Gm,\ldots, \Gm)/2 \xrightarrow[\simeq]{\Phi} A_0(X)/2 \simeq (\Z/2)^{\#\pi_0(X(\R)) -1}.
\]
Together with $K_n^M(\R)/2\simeq\Z/2$, this proves 
\[ 
CH^{d+n}(X,n)/2 \simeq (\Z/2)^{\#\pi_0(X(\R))}.
\]
\end{proof}

\section{Kato homology groups}
\label{sec:KH}

In this section we apply our main results to the study of Kato homology groups. 
After reviewing the Kato conjectures, we establish a connection between higher Chow groups and Kato homology groups. 
The core of this section is a theorem proving that for a smooth projective variety over a $\fB_s(l)$-field $F$ for some prime $l\neq \ch(F)$, the map on Kato homology groups induced by the structure morphism is an isomorphism (\autoref{thm:KH}). 
For the case $l = \ch(F)$, see  \cite{HS25}.

\subsection*{Kato conjectures}
Let $X$ be an excellent scheme over $F$ (\cite[(7.8.5)]{EGAIV}). 
We denote by $X_j$ the set of points $x$ in $X$ with $\dim(\overline{\set{x}}) = j$ and 
by $F(x)$ the residue field at $x$. 
For any $n\ge 0$ and a prime $l$ invertible on $X$, 
we have the following homological complex $KC^{(n)}_{\bullet}(X,\Z/l^r)$ of Bloch--Ogus type (\cite[Proposition~1.7]{Kat86}): 
\[
\xymatrix@C=5mm@R=0mm{
 \cdots \ar[r]^-{\partial} & \displaystyle\bigoplus_{x\in X_{j}}H^{n+j+1}_{l^r}(F(x)) \ar[r]^-{\partial} &\cdots 
 \ar[r]^-{\partial}& \displaystyle\bigoplus_{x\in X_{1}}H^{n+2}_{l^r}(F(x)) \ar[r]^-{\partial} & \displaystyle\bigoplus_{x\in X_{0}}H^{n+1}_{l^r}(F(x)), \\
 & \mbox{\small degree $j$} & & \mbox{\small degree $1$} & \mbox{\small degree $0$} 
}
\]
where 
\[
H_{l^r}^{m+1}(F (x)) := H^{m+1}(F(x), \mu_{l^r}^{\otimes m}). 
\]
The \textbf{Kato homology group} of $X$ (with coefficients in $\Z/l^r$) is defined 
to be the homology group 
\begin{equation}
\label{def:KH}
KH_j^{(n)}(X,\Z/l^r) := H_j(KC_{\bullet}^{(n)}(X,\Z/l^r))
\end{equation}
of the above complex $KC_\bullet^{(n)}(X,\Z/l^r)$. 
In the case of $n=1$ and $X = \Spec(F)$, we have 
\begin{equation}\label{eq:Br}
    KH_0^{(1)}(\Spec(F),\Z/l^r) = H_{l^r}^{2}(F) \simeq \Br(F)[l^r],
\end{equation}
where $\Br(F)[l^r]$ denotes the $l^r$-torsion part of the Brauer group $\Br(F)$ of $F$.

\begin{conj}[{\cite[Conjecture~0.3]{Kat86}}]
\label{conj:Kato}
Let $X$ be a proper smooth connected scheme over a finite field $F$. 
Then, we have 
	\begin{equation}
	\label{eq:KH}
		KH_j^{(0)}(X,\Z/l^r) \xrightarrow{\simeq} \begin{cases}
			\Z/l^r, &j=0, \\
			0, &j>0.
		\end{cases} 
	\end{equation}
\end{conj}
The above conjecture is known for the case $\dim(X)\le 2$ (\cite{Kat86}). 
For a higher dimensional scheme $X$, 
Jannsen and Saito (\cite[Theorem~0.3]{JS09}) proved the above conjecture 
under some assumptions on the resolutions of singularities. 
In particular, the isomorphism \eqref{eq:KH} holds for $j\le 4$ unconditionally. 
When $l\neq \ch(F)$, Kerz--Saito \cite{KS12} completely proved \autoref{conj:Kato}.

Now we assume that the base field $F$ is a local field with finite residue field $k$.  
We assume that there exists a model $\sX$ of $X$ defined over the valuation ring $\cO_F$ of $F$ 
and  its special fiber is denoted by $X_k$.

\begin{conj}[{\cite[Conjecture~5.1]{Kat86}, \cite[Conjecture~B]{JS03}}] \label{conj:Kato2}
Let $F$ be a local field, and suppose that 
$\sX$ is a regular scheme proper flat over $\cO_F$. 
    Then, the residue map
    \begin{equation}\label{eq:res}
        \Delta_j: KH_j^{(1)}(X,\Z/l^r)\to KH_j^{(0)}(X_k,\Z/l^r)
    \end{equation}
    is an isomorphism for all $r\ge 1$ and $j\ge 0$.
\end{conj}
This conjecture has been proved by Kerz--Saito \cite{KS12} when $l$ is invertible on $\sX$.

Next, we suppose that the field $F$ is a global field. 
For any place $v$ of $F$, 
we denote by $F_v$
the completion of $F$ with respect to the place $v$. 

\begin{conj}[{\cite[Conjecture~0.4]{Kat86}, \cite[Conjecture~A]{JS03}}]
\label{conj:Kato3}
Suppose that $F$ is a global field of $\ch(F) = p\ge 0$. 
Let $X$ be a smooth, connected, projective variety over $F$.
Then, for $j>0$ and a prime $l\neq p$, there is an isomorphism 
\[
KH_j^{(1)}(X,\Z/l^r) \xrightarrow{\simeq} \bigoplus_v KH_j^{(1)}(X_{v},\Z/l^r) 
\]	
and, for $j=0$, the short exact sequence below exists:
\[
0\to KH_0^{(1)}(X,\Z/l^r) \to \bigoplus_v KH_0^{(1)}(X_v,\Z/l^r) \to \Z/l^r \to 0, 
\]
where $v$ runs over the set of the places of $F$ and $X_v := X\otimes_F F_v$ for each place $v$. 
\end{conj}

The conjecture above holds when $\dim(X) = 1$ (\cite{Kat86}). 
Moreover, assuming resolution of singularities 
\autoref{conj:Kato3} follows from \cite[Theorem 0.9]{Jan16} (for recent progress, see \cite{Sai10}).

\subsection*{Relation to higher Chow groups}
Let $X$ be a smooth projective and geometrically irreducible scheme of $d = \dim(X)$ over a field $F$. 
By \cite[Section~2]{JS09}, for a prime $l\neq \ch(F)$, 
the structure morphism $f\colon X\to \Spec(F)$ induces a map 
of complexes $KC_{\bullet}^{(n)}(X,\Z/l^r) \to KC_{\bullet}^{(n)}(\Spec(F),\Z/l^r)$ 
and hence a complex 
\begin{equation}\label{eq:RecCpx}
\bigoplus_{x\in X_1} H_{l^r}^{n+2}(F(x)) \xrightarrow{\partial} \bigoplus_{x\in X_0} H_{l^r}^{n+1}(F(x)) \xrightarrow{\Cor} H_{l^r}^{n+1}(F).
\end{equation}
Here, the map $\Cor$ is defined by the corestrictions $\Cor_{F(x)/F}\colon H_{l^r}^{n+1}(F(x)) \to H_{l^r}^{n+1}(F)$ for $x\in X_0$ (cf.~\cite[Appendix A]{JSS14}). 
Taking the homology groups, there is a map 
\begin{equation}\label{eq:fast}
f_{\ast}^{KH}\colon KH_0^{(n)}(X,\Z/l^r) \to KH_0^{(n)}(\Spec(F),\Z/l^r) = H_{l^r}^{n+1}(F). 
\end{equation}

Suppose that $\cd_l(F) =s$ is finite. 
By definition, we have $KH^{(n)}_0(X,\Z/l^r) = 0$ for $n \ge s$. 

\begin{lem}\label{lem:CHKH}
    If we assume $\mu_l \subset F$, 
    then there is a commutative diagram
    \begin{equation}\label{diag:CHKH}
    \vcenter{
    \xymatrix{
    CH^{d+n+1}(X,n+1)/l \ar[r]^-{f_{\ast}^{CH}}\ar[d]^{\simeq} &  K^M_{n+1}(F)/l\ar[d]^{\simeq} \\
    KH_0^{(n)}(X,\Z/l) \ar[r]^-{f_{\ast}^{KH}} & H_l^{n+1}(F), 
    }}
\end{equation}
for any $n\ge 0$, 
where the vertical maps are bijective.
\end{lem}
\begin{proof}
Let $\xi\in X_1$. 
First we briefly recall the construction 
of the boundary map 
\[
\partial = \partial^{(i)}\colon
H^{n+2}(F(\xi),\Z/l(i+1))
\to \bigoplus_{x\in X_0}H^{n+1}(F(x),\Z/l(i))
\]
for $i=n,n+1$. 
Let $C$ be the normalization
of $\overline{\set{\xi}}$. For a closed point $x\in X_0$ with
$x\in \overline{\{\xi\}}$, let $v$ be a closed point of
$C$ lying over $x$. 
We denote by $K_v := F(\xi)_v^h$ the henselization of $F(\xi)$
at the discrete valuation defined by $v$, and by $F(v)$ its residue
field. Then $F(v)$ is a finite extension of $F(x)$. 
Let $K_v^{sh}$ be the
strict henselization. 
Then 
$\partial_v^{(i)}$ is the composite
\begin{align*}	
\partial_v^{(i)}\colon H^{n+2}(K_v,\Z/l(i+1))
&\xrightarrow{\epsilon^{(i)}}
H^{n+1}\!\left(
F(v),
H^1(K_v^{sh},\Z/l(1))\otimes \Z/l(i)
\right) \\
&\xrightarrow{(\rho^{(i)})_\ast}
H^{n+1}(F(v),\Z/l(i)),
\end{align*}
where $\epsilon^{(i)}$ is the edge map of the Hochschild--Serre spectral
sequence and $\rho^{(i)}$ is induced by the normalized valuation
\[
H^1(K_v^{sh},\Z/l(1))
\simeq
(K_v^{sh})^\times/l
\xrightarrow{v} 
\Z/l .
\]
For $i=n,n+1$, define 
\[
\partial^{(i)} 
=\left(
\sum_{v\mid x}
\Cor_{F(v)/F(x)}
\circ
\partial_v^{(i)}
\circ
\res_{K_v/F(\xi)} \right)_{x\in X_0},
\]
where
$v$ 
run through the closed points of
$C$ lying over $x$.

Fix a primitive $l$-th root of unity $\zeta$ in $F$. 
For any field extension $E/F$, 
the norm residue isomorphism theorem (the Bloch--Kato conjecture, \cite[Theorem~6.16]{Vo11} and \cite{Wei09}) gives 
\[
K_{i+1}^M(E)/l \xrightarrow[\simeq]{h^{i+1}_E} 
H^{i+1}(E,\Z/l(i+1)) \xleftarrow[\simeq]{\cup \zeta} 
H^{i+1}(E,\Z/l(i))= H^{i+1}_l(E)
\]
for $i = n, n+1$. 
Take $\xi \in X_1$ and 
consider the following diagram: 
\begin{equation}
\label{diag:cup_zeta}
\vcenter{
\xymatrix@R=7mm{
K_{n+2}^M(F(\xi))/l \ar[r]^{\partial}\ar[d]_{h_{F(\xi)}^{n+2}}^{\simeq} & \displaystyle \bigoplus_{x\in X_0} K_{n+1}^M(F(x))/l\ar[d]_{\simeq}^{\oplus h_{F(x)}^{n+1}} \\
H^{n+2}(F(\xi),\Z/l(n+2)) \ar[r]^-{\partial^{(n+1)}} & \displaystyle \bigoplus_{x\in X_0}H^{n+1}(F(x),\Z/l(n+1))  \\
H_l^{n+2}(F(\xi)) \ar[r]^-{\partial^{(n)}}\ar[u]^{\cup \zeta}_\simeq & \displaystyle \bigoplus_{x\in X_0}H_l^{n+1}(F(x)) \ar[u]_{\cup \zeta}^\simeq.
}}
\end{equation}
The norm residue isomorphism is compatible with norm maps on Milnor $K$-groups and corestriction maps in Galois cohomology.
Hence, the upper square is commutative by \cite[Lemma~1.4 (1)]{Kat80}. 
Furthermore, 
the edge map of the Hochschild--Serre spectral sequence 
is compatible with cup products, 
and the valuation map defining $\rho^{(i)}$ is compatible with tensoring by $\zeta$.
Taking the sum over $v\mid x$ and using the compatibility of
corestriction with cup products, we obtain 
that the lower square in \eqref{diag:cup_zeta}  is commutative. 

The isomorphism \eqref{eq:MK-Chow} gives rise to a commutative diagram below:
\[
\xymatrix{
\displaystyle\bigoplus_{x\in X_1}K_{n+2}^M(F(x))/l \ar[r]^{\partial}\ar[d]^{\simeq} & \displaystyle\bigoplus_{x\in X_0}K_{n+1}^M(F(x))/l \ar[r]\ar[d]^{\simeq} & CH^{d+n+1}(X,n+1)/l\ar[r]\ar@{-->}[d]^{\simeq} & 0\\
\displaystyle\bigoplus_{x\in X_1}H_l^{n+2}(F(x)) \ar[r]^{\partial} & \displaystyle\bigoplus_{x\in X_0}H_l^{n+1}(F(x)) \ar[r] & KH_0^{(n)}(X,\Z/l)\ar[r] & 0
}
\]
The right vertical map is bijective and makes the diagram \eqref{diag:CHKH} commutative. 
\end{proof}

\begin{lem}\label{lem:r}
    Let $s\ge 1$, and suppose that $\cd_l(F) \le s$ for a prime $l$. 
    If 
    \[
    f_{\ast}^{KH}\colon KH_0^{(s-1)}(X,\Z/l) \to KH_0^{(s-1)}(\Spec(F),\Z/l) = H_l^{s}(F)
    \]
    is injective (resp.~surjective), 
    then, for every $r\ge 1$, so is 
    \[
    f_{\ast}^{KH}\colon KH_0^{(s-1)}(X,\Z/l^r) \to H_{l^r}^{s}(F).
    \]
\end{lem}
\begin{proof}
For any $n\ge 1$ and a field extension $K/F$, 
there is a long exact sequence 
\begin{align*}
\cdots &\to H^{n-1}(K,\mu_{l^{r+1}}^{\otimes (n-1)}) \to H^{n-1}(K,\mu_l^{\otimes (n-1)}) \\
&\xrightarrow{\delta} H^{n}_{l^r}(K) \to H^{n}_{l^{r+1}}(K) \to H^{n}_{l}(K) \\
&\xrightarrow{\delta} H^{n+1}(K,\mu_{l^{r}}^{\otimes (n-1)}) \to \cdots 
\end{align*}
along with the short exact sequence $0\to \mu_{l^r}^{\otimes (n-1)} \to \mu_{l^{r+1}}^{\otimes (n-1)}\to \mu_l^{\otimes (n-1)}\to 0$.
By the norm residue isomorphism, 
the vertical maps in the diagram 
\[
\xymatrix{
K_{n-1}^{M}(K)/l^{r+1}\ar[r] \ar[d]^{\simeq} & K_{n-1}^M(K)/l\ar[d]^{\simeq}\\
H^{n-1}(K,\mu_{l^{r+1}}^{\otimes (n-1)}) \ar[r] & H^{n-1}(K,\mu_l^{\otimes (n-1)}) 
}
\]
are isomorphisms. Since the top horizontal map is surjective, so is the bottom one. 
Hence, we have an exact sequence 
\begin{equation}
	\label{seq:GC}
 0 \to H^{n}_{l^r}(K) \to H^{n}_{l^{r+1}}(K) \to H^{n}_{l}(K) \xrightarrow{\delta} H^{n+1}(K,\mu_{l^{r}}^{\otimes (n-1)}).
\end{equation}
Applying $K = F(x)$ for $x\in X_1$ and $n = s+1$ to the exact sequence \eqref{seq:GC}, 
and using $\cd_l(F(x)) \le \cd_l(F)+1 \le s+1$, 
there is a short exact sequence 
\[
0 \to H_{l^r}^{s+1}(F(x)) \to H_{l^{r+1}}^{s+1}(F(x)) \to H_l^{s+1}(F(x)) \to 0. 
\]
Next, we apply $K = F(x)$ for $x\in X_0$ and $n = s$ to the exact sequence \eqref{seq:GC}. 
By $\cd_l(F) \le s$, we have 
$H^{s+1}(F(x),\mu_{l^{r}}^{\otimes (s-1)}) = 0$ for $x\in X_0$ 
and there is a short exact sequence 
\[
0 \to H_{l^r}^{s}(F(x)) \to H_{l^{r+1}}^{s}(F(x)) \to H_l^{s}(F(x)) \to 0. 
\]
We obtain a commutative diagram with exact rows:
\begin{equation}
\label{diag:KH}
\vcenter{	
\xymatrix@C=5mm{
0 \ar[r] & \displaystyle\bigoplus_{x\in X_1}H^{s+1}_{l^{r}}(F(x)) \ar[r]\ar[d]^{\partial} & \displaystyle\bigoplus_{x\in X_1}H^{s+1}_{l^{r+1}}(F(x)) \ar[r]\ar[d]^{\partial} &
\displaystyle\bigoplus_{x\in X_1}H^{s+1}_{l}(F(x))\ar[d]^{\partial} \ar[r] & 0\\
0 \ar[r] & \displaystyle\bigoplus_{x\in X_0}H^{s}_{l^{r}}(F(x)) \ar[r] & \displaystyle\bigoplus_{x\in X_0}H^{s}_{l^{r+1}}(F(x)) \ar[r] &
\displaystyle\bigoplus_{x\in X_0}H^{s}_{l}(F(x))\ar[r]  & 0.
}}
\end{equation}
By considering the cokernels of the vertical maps, 
the above diagram induces 
\begin{equation}
	\label{diag:KHH}
\vcenter{	
\xymatrix@C=5mm{
 & KH_0^{(s-1)}(X,\Z/l^r)\ar[r]\ar[d]^{f_{\ast}^{KH}} & KH_0^{(s-1)}(X,\Z/l^{r+1})  \ar[r]\ar[d]^{f_{\ast}^{KH}} &
KH_0^{(s-1)}(X,\Z/l)\ar[d]^{f_{\ast}^{KH}}\ar[r] & 0\,\\
0 \ar[r] & H^{s}_{l^{r}}(F) \ar[r] & H^{s}_{l^{r+1}}(F) \ar[r] &
H^{s}_{l}(F) \ar[r] & 0.
}}
\end{equation}
The assertion follows from the snake lemma and the induction on $r$.
\end{proof}

\begin{lem}
\label{lem:F'}
	Let $r\ge 1$, and 
	let $F'/F$ be a finite separable extension of degree prime to $l$. 
    If we assume the push-forward map 
    \[
        f_{\ast}^{KH}\colon KH_0^{(n)}(X_{F'},\Z/l^r) \to H_{l^r}^{n+1}(F')
    \]
    is injective (resp.~surjective) for some $n\ge 0$, 
    then so is 
    \[
        f_{\ast}^{KH}\colon KH_0^{(n)}(X,\Z/l^r) \to H_{l^r}^{n+1}(F).
    \]
\end{lem}
\begin{proof}
	Consider the following commutative diagram: 
\[
\xymatrix{
KH_0^{(n)}(X_{F'},\Z/l^r) \ar[r]^-{f_{\ast}^{KH}} & H_{l^r}^{n+1}(F') \\
KH_0^{(n)}(X,\Z/l^r) \ar[u]^{\res_{F'/F}}\ar[r]^-{f_{\ast}^{KH}} & H_{l^r}^{n+1}(F)\ar[u]_{\res_{F'/F}},
}\quad 
\xymatrix{
KH_0^{(n)}(X_{F'},\Z/l^r)\ar[d]_{\Cor_{F'/F}} \ar[r]^-{f_{\ast}^{KH}} & H_{l^r}^{n+1}(F')\ar[d]^{\Cor_{F'/F}} \\
KH_0^{(n)}(X,\Z/l^r) \ar[r]^-{f_{\ast}^{KH}} & H_{l^r}^{n+1}(F),
}
\]
where $\res_{F'/F}$ is the restriction map. 
The restriction map on the Kato homology groups is induced by the restriction map 
\[
H^{n+1}_{l^r}(F(x)) \to \bigoplus_{\begin{matrix}x' \in (X_{F'})_0\\ 
\text{over $x$}\end{matrix}}H_{l^r}^{n+1}(F'(x')) 
\]
whereas the map $f_\ast^{KH}$ is given by corestriction. 
Hence the commutativity of the left diagram follows from the Mackey formula for
restriction and corestriction (\cite[Chapter I, Proposition~1.5.6]{NSW08}). 
Both compositions
\begin{gather}
	KH_0^{(n)}(X,\Z/l^r) \xrightarrow{\res_{F'/F}} KH_0^{(n)}(X_{F'},\Z/l^r) \xrightarrow{\Cor_{F'/F}} KH_0^{(n)}(X,\Z/l^r), \\
	H_{l^r}^{n+1}(F) \xrightarrow{\res_{F'/F}} H_{l^r}^{n+1}(F') \xrightarrow{\Cor_{F'/F}} H_{l^r}^{n+1}(F)
\end{gather}
are multiplication by $[F':F]$. Since the groups are annihilated by $l^r$ and $[F':F]$ is prime to $l$, these compositions are bijective. In particular, the restriction maps are injective and the corestriction maps are surjective. Therefore, if the top horizontal map $f_*^{KH}$ in the above commutative diagrams is injective (resp.\ surjective), then so is the bottom horizontal map.
\end{proof}

\begin{thm}
\label{thm:KH}
    Suppose that $F$ is a $\fB_s(l)$-field for some $s\ge 1$ with $l \neq \ch(F)$.
    Let $X$ be a smooth projective and geometrically irreducible scheme over $F$. 
    {If $l\nmid \delta(X/F)$,} then 
    \[
	f_{\ast}^{KH}\colon KH_0^{(s-1)}(X,\Z/l^r) \xrightarrow{\simeq} KH_0^{(s-1)}(\Spec(F),\Z/l^r) = H_{l^r}^{s}(F) 
    \]
    is an isomorphism for every $r\ge 1$.
\end{thm}
\begin{proof}

From \autoref{lem:KatRem1}, we have $\cd_l(F) \le s$. 
By \autoref{lem:r}, 
it is sufficient to show the case $r=1$. 
The base change $X \otimes_F F(\mu_l)$ is smooth projective and geometrically irreducible. 
Moreover, since a zero-cycle on $X$ of degree prime to $l$
remains a zero-cycle of degree prime to $l$ after base change, we have
$l\nmid\delta(X_{F(\mu_l)}/F(\mu_l))$.
By \autoref{lem:F'}, 
we may assume that $\mu_l \subset F$. 
By \autoref{lem:CHKH} for $n = s-1$, 
there is a commutative diagram: 
\[
\xymatrix{
 CH^{d+s}(X,s)/l \ar[r]^-{f_{\ast}^{CH}}\ar[d]^{\simeq} & K^M_{s}(F)/l\ar[d]^{\simeq} \\
KH_0^{(s-1)}(X,\Z/l) \ar[r]^-{f_{\ast}^{KH}} & H_l^{s}(F), 
}
\]
where $d = \dim(X)$ and the vertical maps are bijective. 
From \autoref{thm:main}, 
the top horizontal map is bijective. 
The bottom horizontal map is therefore bijective.
\end{proof}
\begin{rem}
    In \autoref{thm:KH}, when $F$ is a finite field, we have $\delta(X/F)=1$ by \cite{Sou}.
    Therefore, any prime $l$ satisfies $l\nmid \delta(X/F)$.
\end{rem}

Concerning the Kato conjecture (\autoref{conj:Kato}), we have the following corollary. 
\begin{cor}[{Kerz-Saito \cite{KS12}}]\label{cor:KH_finite}
Let $X$ be a smooth projective and geometrically irreducible scheme over a finite field $F$ 
with $l \neq \ch(F)$.
Then, for every $r\ge 1$, we have 
$KH_0^{(0)}(X,\Z/l^r) \xrightarrow{\simeq} \Z/l^r$. 
\end{cor}
\begin{proof}
By \autoref{thm:KH}, we have 
$KH_0^{(0)}(X,\Z/l^r) \xrightarrow{\simeq} KH_0^{(0)}(\Spec(F),\Z/l^r) = H_{l^r}^{1}(F) $. 
The target $H_{l^r}^{1}(F) = H^1(F,\Z/l^r)$ 
is isomorphic to $\Z/l^r$.
\end{proof}

 \begin{cor}[{cf.~\cite[Cor.~2.7]{For16}}]
Suppose that $F$ is an $N$-dimensional local field of $\ch(F) = p\ge 0$ 
    with residue field $k$ and $N\ge 1$, and let $l$ be a prime with $l\neq p$.
    Let $X$ be a smooth projective and geometrically irreducible scheme over $F$. 
    Assume that there exists a regular and proper flat model $\sX$ of $X$ over $\cO_F$ 
    and its special fiber $X_k$ is smooth. 
    If $l\nmid \delta(X/F)$, then the residue map 
    \[
        \Delta_0 \colon KH_0^{(N)} (X,\Z/l^r) \xrightarrow{\simeq} KH_0^{(N-1)}(X_k,\Z/l^r)
    \]
is an isomorphism for all $r\ge 1$.
\end{cor}
\begin{proof}
The field $F$ is $\fB_{N+1}(l)$ (\autoref{prop:Kato}). 
Consider the following commutative diagram:
\[
\xymatrix{
 KH_0^{(N)} (X,\Z/l^r) \ar[d]_{f_{\ast}^{KH}} \ar[r]^-{\Delta_0} & KH_0^{(N-1)}(X_k,\Z/l^r)\ar[d]^{(f_k)_{\ast}^{KH}}\\ 
 H_{l^r}^{N+1}(F) \ar[r] & H_{l^r}^{N}(k),
 }
\]
where 
$f_k\colon X_k \to \Spec(k)$ is the structure map of the special fiber. 
Since specialization of zero-cycles preserves degrees, we have
$\delta(X_k/k)\mid\delta(X/F)$.
Hence the assumption $l\nmid\delta(X/F)$ implies
$l\nmid\delta(X_k/k)$. When $N=1$, this is automatic since
$k$ is finite and $\delta(X_k/k)=1$. Therefore, both vertical
maps are isomorphisms by \autoref{thm:KH}.
By Kato's higher-dimensional local class field theory, 
we have $H^{N+1}_{l^r}(F) \simeq H^N_{l^r}(k) \simeq \Z/l^r$ (cf.~\cite[Section~1.1, Theorem 3]{Kat80}).
The bottom horizontal map in the above diagram is bijective. 
Therefore, the residue map $\Delta_0$ becomes bijective.
\end{proof}

\begin{cor}[{\cite[Thm.~8.3]{KS12}}]
Let $F$ be a global field of $\ch(F) = p\ge 0$ 
and let $l$ be a prime number with $l\neq p$. 
Let $X$ be a smooth projective and geometrically irreducible scheme over $F$. 
In case $\ch(F) = 0$, we additionally assume that $F$ is totally imaginary or $l\neq 2$. 
If $l\nmid \delta(X/F)$, then, for every $r\ge 1$, there is a short exact sequence 
\[
0\to KH_0^{(1)}(X,\Z/l^r) \to \bigoplus_v KH_0^{(1)}(X_v,\Z/l^r) \to \Z/l^r \to 0, 
\]
where $v$ runs over the set of the places of $F$ and $X_v := X\otimes_F F_v$ for each place $v$. 
\end{cor}
\begin{proof}
As noted in \autoref{sec:Mackey}, 
$F$ is a $\fB_2(l)$-field. 
For any finite place $v$ of $F$, 
the local field $F_v$ is also a $\fB_2(l)$-field (\autoref{prop:Kato}). 
For an infinite place $v$ of $F$ (when $\ch(F) = 0$), 
$KH_0^{(1)}(X_v,\Z/l^r)$ is a quotient of $\bigoplus_{x\in (X_v)_0}H_{l^r}^{2}(F_v(x)) = 0$ 
if $l>2$. 
Let $f_v\colon X_v \to \Spec(F_v)$ be the base change of $f\colon X\to \Spec(F)$ for each place $v$ of $F$.
Since $l\nmid \delta(X/F)$, we also have
$l\nmid \delta(X_v/F_v)$ for every place $v$. Indeed, a zero-cycle
on $X$ of degree prime to $l$ gives, after base change to $F_v$, a zero-cycle on $X_v$ of degree prime to $l$.
From \autoref{thm:KH}, 
 we have isomorphisms 
\begin{align*}
    f_{\ast}^{KH}\colon KH_0^{(1)}(X,\Z/l^r) &\xrightarrow{\simeq} KH_0^{(1)}(\Spec(F),\Z/l^r) \simeq \Br(F)[l^r],\quad \mbox{and}\\
    (f_v)_{\ast}^{KH}\colon KH_0^{(1)}(X_v,\Z/l^r) &\xrightarrow{\simeq }KH_0^{(1)}(\Spec(F_v),\Z/l^r) \simeq \Br(F_v)[l^r] 
\end{align*}
for every finite place $v$ of $F$.
Considering the facts $\Br(\mathbb{C}) = 0$ and $\Br(\R)\simeq \Z/2$, 
the required short exact sequence is nothing other than the Hasse--Brauer--Noether theorem: 
\[
0 \to \Br(F)[l^r] \to \bigoplus_v \Br(F_v)[l^r] \to \Z/l^r \to 0.
\]
\end{proof}

Finally, we discuss the Kato homology groups of a field that is not an arithmetic field.
\begin{ex}
Let $F = \C(\!(t)\!)$ be the field of Laurent series over $\C$. 
The field $F$ is a $\fB_1$-field (\autoref{prop:cdvf}). 
For any prime $l$, we have 
\[
K_1^M(F)/l^r \simeq K_1(\C)/l^r \oplus K_0^M(\C)/l^r = \Z/l^r
\]
(\cite[Chapter III, Corollary 7.3.1]{Wei13}).
For a smooth projective and geometrically irreducible scheme $X$ of dimension $d$ over $F$ with $\delta(X/F) = 1$, 
we have 
$CH^{d+1}(X,1)/l^r \simeq \Z/l^r$ by \autoref{cor:main}. 
Every finite extension of $F$ is cyclic and is of the form
$\C(\!(t^{1/m})\!)$ for some $m\ge1$. Consequently,
the absolute Galois group of $F$ is procyclic:
$G_F \simeq \Zhat$. 
\autoref{thm:KH} gives 
\[
KH_0^{(0)}(X,\Z/l^r) \xrightarrow{\simeq} H^1(F,\Z/l^r) \simeq \Hom_{\mathrm{cont}}(\Zhat,\Z/l^r) \simeq \Z/l^r.
\]
\end{ex}

\def\cprime{$'$}
\providecommand{\bysame}{\leavevmode\hbox to3em{\hrulefill}\thinspace}


\end{document}